\documentclass[11pt]{amsart}
\usepackage{amsmath,amsthm,amscd,amsfonts,amssymb,epic,eepic,bbm, graphicx,ytableau}
\usepackage{color,graphicx,tikz-cd,xcolor}
\usepackage{shuffle}
\usepackage{tikz}
\usepackage{tkz-euclide}
\usetikzlibrary{angles,quotes}
\usetkzobj{all}
\usetikzlibrary{snakes}
\usetikzlibrary{shapes,calc,decorations}
\usepackage[misc]{ifsym}
\numberwithin{equation}{section}

\allowdisplaybreaks

\newtheorem{thm}{Theorem}[section]

\newtheorem{conj}[thm]{Conjecture}
\newtheorem{lem}[thm]{Lemma}
\newtheorem{defn}[thm]{Definition}

\newtheorem{prop}[thm]{Proposition}

\newtheorem{rem}[thm]{Remark}
\newtheorem{example}[thm]{Example}
\theoremstyle{remark}

\newcommand\inv{\operatorname{inv}}
\newcommand\SYT{\operatorname{SYT}}
\newcommand\SSYT{\operatorname{SSYT}}
\newcommand\LLT{\operatorname{LLT}}

\newcommand\bdnu{\boldsymbol{\nu}}
\newcommand\bx{{\bf x}}
\newcommand\bdT{\boldsymbol{T}}
\definecolor{pink1}{RGB}{219, 48, 122}

\newcommand{\la}{\lambda}

\definecolor{dredcolor}{rgb}{0.9,0.3,0.4}

\begin{document}

\title[Chromaric quasisymmetric functions and LLT polynomials]{Melting lollipop chromatic quasisymmetric functions\\ and Schur expansion of unicellular LLT polynomials}

\author{JiSun Huh}
\address{Department of Mathematics, Ajou University, Suwon 16499 Republic of Korea}
\email{hyunyjia@ajou.ac.kr}

\author{Sun-Young Nam}
\address{Department of Mathematics, Sogang University, Seoul 04107, Republic of Korea}
\email{synam.math@gmail.com}

\author{Meesue Yoo}
\address{
Applied Algebra and Optimization Research Center, Sungkyunkwan University, Suwon 16420,
Republic of Korea}
\email{meesue.yoo@skku.edu (\Letter)}


\thanks{
The first author was supported by NRF grant \#2015R1D1A1A01057476.
The second author was supported by NRF grant \#2017R1D1A1B03030945. 
The third author was supported by NRF grants \#2016R1A5A1008055 and \#2017R1C1B2005653.}

\begin{abstract} 
In this work, we generalize and utilize the linear relations of LLT polynomials introduced by Lee \cite{Lee}. By using the fact that the chromatic quasisymmetric functions 
and the unicellular LLT polynomials are related via plethystic substitution and thus they satisfy the same linear relations, we can apply the linear relations to both sets of functions. 

As a result, in the chromatic quasisymmetric function side, we find a class of $e$-positive graphs, called \emph{melting lollipop graphs}, and explicitly prove the $e$-unimodality. In the unicellular LLT side, we obtain Schur expansion formulas for LLT polynomials corresponding to certain set of graphs, namely, complete graphs, path graphs, lollipop graphs and melting lollipop graphs. 
\end{abstract}

\maketitle

\tableofcontents


\section{Introduction}

Shareshian and Wachs \cite{SW} introduced the \emph{chromatic quasisymmetric function} as a refinement of Stanley's chromatic symmetric function introduced in \cite{S1} by considering an extra parameter $q$. Recall that a \emph{proper coloring} of a simple graph $G=(V,E)$ is any function $\kappa : V \rightarrow \{1,2,3,\dots\}$
satisfying that $\kappa(u)\neq \kappa(v)$ for any $u, v\in V$ such that $\{u,v\}\in E$.  For a simple graph $G$ with a vertex set $V$ of positive integers and  a proper coloring $\kappa$,  we denote the number of edges $\{i,j\}$ of $G$ with $i<j$ and $\kappa(i)<\kappa(j)$ by ${\rm asc}(\kappa)$. Given  a sequence $\bx=(x_1,x_2,\dots)$ of commuting indeterminates, the chromatic quasisymmetric function of $G$ is defined as
\begin{equation}\label{eqn:quasichro}
X_G({\bf x};q)=\sum_{\kappa}q^{\rm{asc}(\kappa)} x^{\kappa},
\end{equation}
where the sum is over all proper colorings $\kappa$. The function $X_G(\bx;1)$ is called the \emph{chromatic symmetric function} and denoted by $X_G(\bx)$.

Shareshian and Wachs showed that if $G$ is the incomparability graph of a natural unit interval order, then the coefficients of $q^i$ in $X_G({\bf x};q)$ are symmetric functions and form a palindromic sequence. They also made a conjecture (Conjecture \ref{conj:SW}) on the $e$-positivity and the $e$-unimodality of $X_G({\bf x};q)$, which specializes to the famous $e$-positivity conjecture on the chromatic symmetric functions of Stanley and Stembridge \cite{S1,SS}. In \cite{Guay}, Guay-Paquet proved that if Conjecture \ref{conj:SW} holds, then Stanley and Stembridge's conjecture also holds. This result has put a spotlight on the incomparability graphs of natural unit interval orders. As a result, the $e$-positivity (and $e$-unimodality) of several subclasses of the incomparability graphs of natural unit interval orders are proved, see \cite{AP,CH,Dah,DW,FHM,HP,SW}. 

On the other hand, if we remove the proper condition of the colorings, the sum \eqref{eqn:quasichro} over all colorings $\kappa$ gives \emph{the unicellular LLT polynomials}. The LLT polynomials are defined by Lascoux, Leclerc and Thibon in \cite{LLT} 
which can be considered as a $q$-deformation of a product of Schur functions. The LLT polynomials are indexed by a tuple of skew diagrams, but especially when each skew diagram consists of only one cell, then the corresponding LLT polynomial is called \emph{unicellular}. In \cite{GH}, Grojnowski and Haiman proved the Schur positivity of LLT polynomials using the Kazhdan--Lusztig theory, but there is no known combinatorial description for the Schur coefficients, except some special cases. In a sense that the combinatorial description for the Schur coefficients of LLT polynomials could give a combinatorial description of the $q,t$-Kostka polynomials which are the Schur coefficients of the modified Macdonald polynomials,
having a combinatorial formula for the Schur expansion of the LLT polynomials is very important. In the case of the unicellular LLT polynomials, due to the correspondence between the unicellular LLT diagrams and the Dyck diagrams (cf. \cite{AP}, \cite{Lee}), abundant connections to other branches of mathematics such as Hopf algebra and Hessenberg varieties have been figured out. 


The precise relationship between chromatic quasisymmetric functions and the unicellular LLT polynomials is given by Carlsson and Mellit \cite{CM} via plethystic substitution. This relationship explains the parallel phenomena between the chromatic quasisymmetric functions and the unicellular LLT polynomials. In particular, the two sets of polynomials satisfy the same linear relations.

Recently,
Lee in \cite{Lee} introduced local linear relations on LLT polynomials and used them to prove the $k$-Schur positivity of the LLT polynomials when $k=2$. In this work, we generalize and utilize these local linear relations.
In chromatic quasisymmetric function side, by using these linear relations, we find a class of $e$-positive graphs, called \emph{melting lollipop graphs}. In unicellular LLT side, we prove Schur expansion formulas for LLT diagrams corresponding to certain set of graphs. 

The contents of the paper is organized as follows. 
In Section \ref{sec:pre} we provide necessary definitions and known results used throughout the paper. 
In Section \ref{sec:local} we generalize and utilize those local linear relations. 
By using these linear relations, in section \ref{sec:lollipopG}
we explain $e$-positivity and $e$-unimodality of the chromatic quasisymmetric functions corresponding to lollipop graphs.
And, we find a new class of 
$e$-positive and $e$-unimodal graphs, called melting lollipop graphs.
Thereafter we obtain a combinatorial interpretation for Schur expansion of unicellular LLT polynomials corresponding to melting lollipop graphs in section \ref{sec:LLT_Schur}.
To do this, 
we first obtain Schur expansion formulas for unicellular LLT polynomials corresponding to certain set of basic graphs including complete graphs and path graphs.
In the final section, as an aside, 
we introduce a  combinatorial  way  to  compute  the  Schur  coefficients  of  LLT polynomials when the Schur functions are indexed by hook shapes.




\section{Preliminary}\label{sec:pre}


In this section
we collect definitions and notions which are required to develop our arguments.
More details can be found in \cite{AP,HHL05, SW}.


\subsection{Semistadard Young tableaux and Schur functions}


A \emph{partition} of $n$ is a nonincreasing sequence $\la=(\la_1, \la_2, \dots, \la_\ell)$ of positive integers such that $\sum_i \la_i =n$ and we use the notation $\la\vdash n$ to denote that $\la$ is a partition of $n$. Each $\lambda_i>0$ is called a \emph{part} of $\lambda$, and we define the \emph{length} $\ell (\lambda)$ of $\lambda$ 
to be the number of parts in $\lambda$. 

When two partitions $\lambda$ and $\mu$ satisfy $\mu_i \leq \lambda_i$ for all $i \geq 1$, 
we write $\mu \subseteq \lambda$,
and define  \emph{skew shape} 
$\lambda/\mu$ to be the set theoretic difference $\lambda\setminus \mu$.
The \emph{diagram} of $\lambda/\mu$ is defined to be the set $\{ (i,j)~ :~ 1\le i\le \ell(\lambda), ~ \mu_i < j\le \lambda_i\}$.
Throughout this paper,
we frequently identify skew shape $\lambda/\mu$ with its diagram,
and the elements in the diagram of $\lambda/\mu$ are visualized as boxes in a plane.


\begin{figure}[ht]
\begin{tikzpicture}[scale=.48]
\draw[thick] (0,4)--(2,4)--(2,3)--(4,3)--(4,1)--(5,1)--(5,0)--(4,0)--(4,1)--(2,1)--(2,2)--(1,2)--(1,3)--(0,3)--(0,4)--cycle;
\draw[thick] (1,3)--(1,4);
\draw[thick] (2,2)--(2,3);
\draw[thick] (3,1)--(3,3);
\draw[thick] (1,3)--(2,3);
\draw[thick] (2,2)--(4,2);
\draw[dashed] (0,0)--(0,3);
\draw[dashed] (0,0)--(4,0);
\draw[dashed] (0,2)--(1,2);
\draw[dashed] (0,1)--(2,1);
\draw[dashed] (1,0)--(1,2);
\draw [dashed] (2,0)--(2,1);
\draw [dashed] (3,0)--(3,1);
\end{tikzpicture}
\caption{The skew shape $(5,4,4,2)/(4,2,1)$.}
\end{figure}
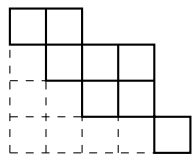
Let $\nu$ be a skew shape $\lambda/\mu$. The \emph{content} of a cell $u=(i,j)$ in $\nu$ is the integer $c(u)=i-j$. 
A \emph{(semistandard) Young tableau} of shape $\nu$ is a filling of $\nu$ with letters from $\mathbb{Z}_+$
such that 
the entries are weakly increasing on each row and strictly increasing on each column. 
For a Young tableau $T$,
we define the \emph{weight} $\text{wt}(T)$ of $T$ to be the sequence $(\text{wt}_1(T), \text{wt}_2(T), \cdots, )$,
where $\text{wt}_i(T)$ denotes the number of occurrences of $i$ in $T$.
Especially,
$T$ is called \emph{standard} if its weight is $(1,1, \cdots, 1)$.
We denote the set of Young tableaux of shape $\nu$ by $\SSYT (\nu)$ and
the set of standard Young tableaux of shape $\nu$ by $\SYT(\nu)$.
Given $T\in \SSYT(\nu)$, define its corresponding monomial 
$x^T = \prod_{u\in\nu} x_{T(u)}.$
Then the \emph{Schur functions} are defined as follows:
\[
s_\lambda = \sum_{T \in \SSYT(\lambda)} {x}^T \, .
\]
For partitions $\lambda, \mu$ and  $\nu$,
the \emph{Littlewood-Richardson (LR) coefficients} $c_{\mu}^{\nu/\lambda}$ are the structure constants
appearing in the Schur expansion of the product of two Schur functions $s_\lambda$ and $s_\mu$,
that is, 
\[
s_\lambda \cdot s_\mu = \sum_\nu c_\mu^{\nu/\lambda}s_\nu \, .
\]
It is a well-established fact that
the LR coefficients $c_\mu^{\nu/\lambda}$ are nonnegative integers
and their nonnegativity can be interpreted in a combinatorial way 
by virtue of
\emph{Sch\"{u}zenberge's jeu de taquin} sliding process.
Note that this sliding process is a combinatorial algorithm taking a Young tableau to another Young tableau with the same weight, 
but different shape, that can be carried out by applying the slide described in Figure \ref{jeu de taquin} in succession.
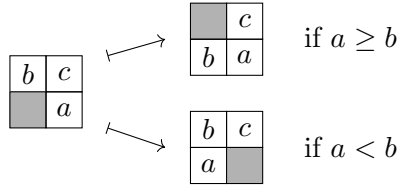
\begin{figure}[h!]
\centering
\begin{tikzpicture}[baseline=0pt]
\def \hhh{4.8mm}      
\def \vvv{4.8mm}      
\def \vvvv{7mm}   
\def \vvvvv{-4mm}    
\draw[-] (0,0-\vvvv) rectangle (\hhh*1,\vvv*1-\vvvv);
\draw[-] (\hhh*1,0-\vvvv) rectangle (\hhh*2,\vvv*1-\vvvv);
\draw[fill=black!30] (\hhh*0,-\vvv*1-\vvvv) rectangle (\hhh*1,0-\vvvv);
\draw[-] (\hhh*1,-\vvv*1-\vvvv) rectangle (\hhh*2,0-\vvvv);
\draw[|->] (\hhh*2.7,\vvv*1-\vvvv) to (\hhh*4.3,0);
\draw[|->] (\hhh*2.7,-\vvv*1-\vvvv) to (\hhh*4.3,-\vvv*3);
\draw[-] (\hhh*6,0) rectangle (\hhh*7,\vvv*1);
\draw[fill=black!30] (\hhh*0+\hhh*5,0) rectangle (\hhh*1+\hhh*5,\vvv*1);
\draw[-] (\hhh*0+\hhh*5,-\vvv*1) rectangle (\hhh*1+\hhh*5,0);
\draw[-] (\hhh*1+\hhh*5,-\vvv*1) rectangle (\hhh*2+\hhh*5,0);
\draw[-] (\hhh*0+\hhh*5,0-\vvv*3) rectangle (\hhh*6,\vvv*1-\vvv*3);
\draw[-] (\hhh*0+\hhh*6,0-\vvv*3) rectangle (\hhh*7,\vvv*1-\vvv*3);
\draw[-] (\hhh*5,-\vvv*1-\vvv*3) rectangle (\hhh*6,0-\vvv*3);
\draw[fill=black!30] (\hhh*6,-\vvv*1-\vvv*3) rectangle (\hhh*7,-\vvv*3);
\node (A) at (\hhh*5.5,0) {};
\node () [right=of A]
{if $a \geq b$};
\node at (\hhh*5.5,-\vvv*3) (A21) {};
\node () [right=of A21]
{if $a < b$};
\node[] at (\hhh*0.5,\vvv*0.5-\vvvv) () {$b$};
\node[] at (\hhh*1.5,\vvv*0.5-\vvvv) () {$c$};
\node[] at (\hhh*1.5,-\vvv*0.5-\vvvv) () {$a$};
\node[] at (\hhh*1.5+\hhh*5,\vvv*0.5) () {$c$};
\node[] at (\hhh*0.5+\hhh*5,-\vvv*0.5) () {$b$};
\node[] at (\hhh*1.5+\hhh*5,-\vvv*0.5) () {$a$};
\node[] at (\hhh*0.5+\hhh*5,-\vvv*2.5) () {$b$};
\node[] at (\hhh*1.5+\hhh*5,-\vvv*2.5) () {$c$};
\node[] at (\hhh*0.5+\hhh*5,-\vvv*3.5) () {$a$};
\end{tikzpicture}
\caption{The jeu de taquin slide} \label{jeu de taquin}
\end{figure}
Given a Young tableau $T$ of skew shape, 
one can apply jeu de taquin slides to $T$ iteratively to get a Young tableau of partition shape, 
which is called the \emph{rectification} of $T$ and denoted by $\text{Rect}(T)$.
Let $ R_\lambda$ be the standard Young tableau of shape $\lambda$,
called the row tableau of shape $\lambda$, 
whose $i$th row consists of the entries 
$\lambda_0 + \cdots + \lambda_{i-1} + 1, \, \lambda_0 + \cdots + \lambda_{i-1} + 2, \,
\cdots , \, \lambda_0 + \cdots + \lambda_{i-1} + \lambda_i $ from left to right for every
$1 \leq i \leq \ell$. 
Here $\lambda_0$ is set to be 0.
Then the LR coefficient
$c_\mu^{\nu/\lambda}$ equals the number of standard Young tableaux $T$ of shape $\nu/\lambda$ a
satisfying that $\text{Rect}(T) = R_\mu$.

\subsection{Natural unit interval orders} \label{sec:nui}

For a positive integer $n$, we set $[n]=\{1,2,\dots,n\}$ and $[n]_q=1+q+\cdots+q^{n-1}$. Let $P$ be a poset. The \emph{incomparability graph} ${\rm inc}(P)$ of $P$ is a graph which has as vertices the elements of $P$, with edges connecting pairs of incomparable elements.  
For a list ${\bf m}:=(m_1, m_2, \dots, m_{n-1})$ of nondecreasing positive integers satisfying that $i\leq m_i \leq n$ for each $i\in[n-1]$, the corresponding \emph{natural unit interval order} $P(\bf{m})$ is the poset on $[n]$ with the order relation given by $i<_{P(\bf{m})} j$ if $i<n$ and $j\in \{m_i +1, m_i +2, \dots, n\}$. 
For example, Figure \ref{fig:lollipop} shows the graph ${\rm inc}(P)$ for the natural unit interval order $P=P(2,3,4,5,6,11,11,11,11,11)$.

\begin{figure}[ht]
\begin{tikzpicture}[scale=.18]
\foreach \n in {6}
 \foreach \i in {1,...,\n}
    \fill (\i*360/\n:\n) coordinate (n\i) circle(2*\n pt)
      \ifnum \i>1 foreach \j in{\i,...,1}{(n\i)edge(n\j)}\fi;
\draw (180:6)--(-31,0);

\foreach \i in {1,...,5}
 \node[below] at (-36+5*\i,0) {\i}; 

\node[below] at (180:6) {6};
\node[below] at (240:6) {7};
\node[below] at (300:6) {8};
\node[right] at (0:6) {9};
\node[above] at (60:6) {10};
\node[above] at (120:6) {11};

\filldraw [black] (-11,0) circle (12pt)
			(-16,0) circle (12pt)
			(-21,0) circle (12pt)
			(-26,0) circle (12pt)
			(-31,0) circle (12pt);
\end{tikzpicture}
\caption{The incomparability graph $L_{6,5}$ of  $P=P(2,3,4,5,6,11,11,11,11,11)$.}\label{fig:lollipop}
\end{figure}
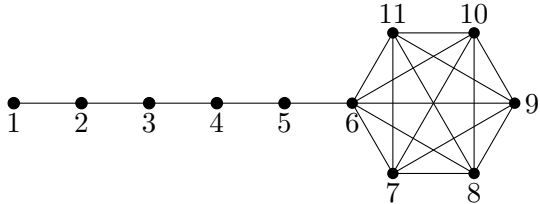

It is well-known that the set of \emph{elementary symmetric functions} $\{e_\la \,:\, \la\vdash n\}$ is a basis of the subspace $\Lambda^n$ of symmetric functions of degree $n$. 
 One says a symmetric function $f({{\bf x}})\in\Lambda^n $ is $b$-{\it positive} if the expansion of $f({\bf x})$ in the basis $\{b_{\la}\}$ has nonnegative coefficients when $\{b_\la\}$ is a basis of $\Lambda^n$. One of main motivation for our study is to understand the following conjecture.

\begin{conj}[Shareshian-Wachs \cite{SW}]\label{conj:SW}
For the incomparability graph $G$ of a natural unit interval order, if $X_G({\bf x};q)=\sum_{i=0}^{m}a_i({\bf x})q^i$ then $a_{i}({\bf x})$ is $e$-positive for all $i$, and $a_{i+1}({\bf x})-a_{i}({\bf x})$ is $e$-positive whenever $0\leq i < \frac{m-1}{2}$.
\end{conj}

We denote the complete graph with $m$ vertices by $K_m$ and the path with $n$ vertices with natural labelling by $P_n$.
For each of $K_{m}$ and $P_{n}$, it is proved that the above conjecture is true.

\begin{prop}\cite[Table 1]{SW}\label{prop:basic}
Let $m$ and $n$ be nonnegative integers.
\begin{enumerate}
\item[(a)] The chromatic quasisymmetric function of a complete graph $K_m$ is 
$$X_{K_m}({\bf x};q)=[m]_q!e_m.$$
\item[(b)] The chromatic quasisymmetric function of a path $P_n$ is
$$X_{P_n}({\bf x};q)=\sum_{m=1}^{\lfloor \frac{n+1}{2} \rfloor} \sum_{\substack{k_1,\dots,k_m \geq 2\\ \sum k_i=n+1}} e_{(k_1-1,k_2,\dots,k_m)} q^{m-1} \prod_{i=1}^{m}[k_i-1]_q.$$
\end{enumerate}
\end{prop}



\subsection{LLT polynomials}

LLT polynomials are a family of symmetric functions introduced by Lascoux, Leclerc and Thibon in \cite{LLT} which naturally arise in the 
description of the power-sum plethysm operators on symmetric functions. The original definition of LLT polynomials uses \emph{cospin} statistic 
of ribbon tableaux, but Haiman and Bylund found a consistent statistic, called \emph{inv}, defined over $n$-tuples of semistandard Young tableaux of various skew shapes. 
Here, we use the inversion statistic of Haiman and Bylund to define the LLT polynomials used in this paper. For the proof of the consistency of two definitions, 
see \cite{HHLRU}.

We consider a $k$-tuple of skew diagrams $\bdnu = (\nu^{(1)},\dots, \nu^{(k)})$
and let 
$$\SSYT(\bdnu) = \SSYT(\nu^{(1)})\times \cdots \times\SSYT(\nu^{(k)}). $$
Given $\bdT\in (T^{(1)},\dots, T^{(k)})\in\SSYT (\bdnu)$, we set 
$$x^{\bdT}=\prod_{i=1}^k x^{T^{(i)}}.$$
We define the \emph{content reading order} by giving the ordering on the cells in $\bigsqcup\bdnu$, so that the left-bottom most cell has the 
smallest order and the order increases upward along diagonals, moving from the left to the right. Given $\bdT\in (T^{(1)},\dots, T^{(k)})\in\SSYT (\bdnu)$,
we obtain the \emph{content reading word} by reading the entries according to the content reading order. 
We say a pair of entries $T^{(i)}(u) > T^{(j)}(v)$ form an \emph{inversion} if either 
\begin{itemize}
\item[(i)] $i<j$ and $c(u)=c(v)$, or
\item[(ii)] $i> j$ and $c(u)=c(v)+1$.
\end{itemize}
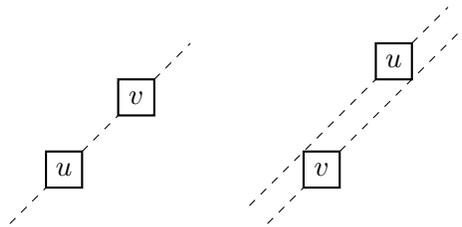
\begin{figure}[ht]
\begin{tikzpicture}[scale=.48]
\draw[thick] (0,0)--(1,0)--(1,1)--(0,1)--(0,0)--cycle;
\draw[thick] (2,2)--(3,2)--(3,3)--(2,3)--(2,2)--cycle;
\node (0) at (.5,.5) {$u$};
\node (1) at (2.5,2.5) {$v$};
\draw[dashed] (-1,-1)--(0,0);
\draw[dashed] (1,1)--(2,2);
\draw[dashed] (3,3)--(4,4);
\end{tikzpicture}\qquad
\begin{tikzpicture}[scale=.48]
\draw[thick] (0,0)--(1,0)--(1,1)--(0,1)--(0,0)--cycle;
\draw[thick] (2,3)--(3,3)--(3,4)--(2,4)--(2,3)--cycle;
\draw[dashed] (-1,-1)--(0,0);
\draw[dashed] (1,1)--(4.5,4.5);
\draw[dashed] (-1.5,-.5)--(2,3);
\draw[dashed] (3,4)--(4,5);
\node (0) at (.5,.5) {$v$};
\node (1) at (2.5,3.5) {$u$};
\end{tikzpicture}
\caption{The locations of cells $u$ and $v$ forming inversion pair of the cases ${\rm (i)}$ and ${\rm (ii)}$ from the left.}
\end{figure}
Let $\inv(\bdT)$ be the number of inversions in $\bdT$.

\begin{defn}
The LLT polynomial indexed by $\bdnu= (\nu^{(1)},\dots, \nu^{(k)})$ is 
$$\LLT_{\bdnu}({\bf x};q) = \sum_{\bdT\in\SSYT(\bdnu)}q^{\inv (\bdT)}  x^{\bdT}.$$
\end{defn}

%
\subsection{Relation between the LLT polynomials and chromatic quasisymmetric functions}\label{subsec:relation}

From now on, we only consider unicellular LLT polynomials, i.e., each $\nu^{(i)}$ consists of only one cell. 
In \cite{AP} and independently in \cite{Lee}, bijective correspondence between unicellular LLT diagrams and skew shapes contained in a staircase shape partition 
has been introduced. 
We explain it with an example. 

\begin{example}\label{ex:dg}
In Figure \ref{fig:example_nu}, the labelling in the left side figure denotes the content reading order of the given LLT diagram $\bdnu$ and we identify those numbers with the numbers in the main diagonal in the figure on the right hand side.

\begin{figure}[ht]
\begin{tikzpicture}[scale=.46]
\draw[thick] (1,2)--(2,2)--(2,3)--(1,3)--(1,2)--cycle;
\draw[thick] (3,3)--(4,3)--(4,4)--(3,4)--(3,3)--cycle;
\draw[thick] (5,4)--(6,4)--(6,5)--(5,5)--(5,4)--cycle;
\draw[thick] (6,7)--(7,7)--(7,8)--(6,8)--(6,7)--cycle;
\draw[thick] (8,8)--(9,8)--(9,9)--(8,9)--(8,8)--cycle;
\draw[thick] (10,9)--(11,9)--(11,10)--(10,10)--(10,9)--cycle;
\draw[thick] (11,11)--(11,12)--(12,12)--(12,11)--(11,11)--cycle;
\draw [dashed] (0.3,1.3)--(1,2);
\draw [dashed] (2,3)--(6,7);
\draw [dashed] (7,8)--(12.5,13.5);
\draw [dashed] (.8,.8)--(3,3);
\draw [dashed] (4,4)--(8,8);
\draw [dashed] (9,9)--(11,11);
\draw [dashed] (12,12)--(13,13);
\draw [dashed] (1.3,0.3)--(5,4);
\draw [dashed] (6,5)--(10,9);
\draw [dashed] (11,10)--(13.5,12.5);
\node (1) at (1.5, 2.5) {$1$};
\node (2) at (6.5, 7.5) {$2$};
\node (3) at (3.5, 3.5) {$3$};
\node (4) at (8.5, 8.5) {$4$};
\node (5) at (11.5, 11.5) {$5$};
\node (6) at (5.5, 4.5) {$6$};
\node (7) at (10.5, 9.5) {$7$};
\node (8) at (9.5, 4) {$\bdnu$};
\node (9) at (13, 6) {$\Longleftrightarrow$};
\end{tikzpicture}
\qquad
\begin{tikzpicture}[scale=.6]
\draw (0,0)--(0,7)--(7,7);
\foreach \i in {1,...,6}
\draw (0,\i)--(\i,\i);
\foreach \i in {1,...,6}
\draw (\i,7)--(\i,\i);
\node (1) at (.5, .5) {$1$};
\node (2) at (1.5, 1.5) {$2$};
\node (3) at (2.5, 2.5) {$3$};
\node (4) at (3.5, 3.5) {$4$};
\node (5) at (4.5, 4.5) {$5$};
\node (6) at (5.5, 5.5) {$6$};
\node (7) at (6.5, 6.5) {$7$};
\node (8) at (.5, 6.5) {$\mathsf{X}$};
\node (8) at (1.5, 6.5) {$\mathsf{X}$};
\node (8) at (2.5, 6.5) {$\mathsf{X}$};
\node (8) at (3.5, 6.5) {$\mathsf{X}$};
\node (8) at (.5, 5.5) {$\mathsf{X}$};
\node (8) at (1.5, 5.5) {$\mathsf{X}$};
\node (8) at (2.5, 5.5) {$\mathsf{X}$};
\node (8) at (.5, 4.5) {$\mathsf{X}$};
\node (8) at (1.5, 4.5) {$\mathsf{X}$};
\node (8) at (.5, 3.5) {$\mathsf{X}$};
\node (8) at (1.5, 3.5) {$\mathsf{X}$};
\node (8) at (.5, 2.5) {$\mathsf{X}$};
\draw[very thick, color=dredcolor] (0,1)--(0,2)--(1,2)--(1,3)--(2,3)--(2,5)--(3,5)--(3,6)--(4,6)--(4,7)--(6,7)--(6,6)--(5,6)--(5,5)--(4,5)--(4,4)--(3,4)--(3,3)--(2,3)--(2,2)--(1,2)--(1,1)--(0,1)--cycle;
\node (9) at (5,2) {$\pi_{\bdnu}$};
\node (10) at (2,-.7) {$\quad$};
\end{tikzpicture}
\caption{}\label{fig:example_nu}
\end{figure}

Starting from the top row (or the largest possible number in the reading order), we cross out the cells in the same row corresponding to the numbers which cannot make an inversion pair 
with the given number of the row. We obtain a (top-left justified) skew shape contained in a staircase shape in the end. 
\end{example}

We call the skew shape coming from an LLT diagram $\bdnu$ the \emph{Dyck diagram} and denote it by $\pi_{\bdnu}$, since the outer boundary defines a Dyck path.
Reading the number of cells in the $i$th row from the top defines an \emph{area sequence} which has the last entry $0$. 
For instance, the area sequence of the Dyck diagram in Figure \ref{fig:example_nu} is $a=(2,2,2,1,1,1,0)$. We denote the area sequence
of the Dyck diagram corresponding to the LLT diagram $\bdnu$ by $a_{\bdnu}$.


%


Given a Dyck diagram corresponding to an LLT diagram $\bdnu$, we can associate a graph which has $V=\{\nu^{(i)} :~ 1\le i \le n\}$ 
as the vertex set and the edge set as follows : label the vertices $\nu^{(i)}$ according to the reverse of the content reading order and for $i<j$, $(i,j)\in E$ 
if the cell in the column labeled $i$ and row labeled $j$ is contained in the Dyck diagram. In other words, replace the diagonal entries in the Dyck diagram $i$ by $n+1-i$, and according to this labelling, $(i,j)\in E$ if the cell in the row $i$ and column $j$ is in the Dyck diagram. 
We denote this graph by $G_{\bdnu}$. We also use the notation $\bdnu_G$ for the LLT diagram corresponding to a graph $G$. 
If there is no confusion, we use the notation $\LLT_{G}({\bf x};q)$ for $\LLT_{\bdnu_G}({\bf x};q)$.

\begin{example}

We keep considering the LLT diagram given in Figure \ref{fig:example_nu}.
To obtain the graph corresponding to the LLT diagram $\bdnu$, we relabel the main diagonal in reverse order (as in the left hand side of Figure \ref{fig:example_G}), and then draw edges $(i,j)$ for $i<j$ if the cell in the row $i$ and column $j$ is contained in the Dyck diagram $\pi_{\bdnu}$. Note that the $i$th cell in $\bdnu$ in terms of the reading order corresponds to the $n+1-i$ vertex in $G_{\bdnu}$.

\begin{figure}[ht]
\begin{tikzpicture}[scale=.6]
\fill[fill=red!8!white] (0,1) rectangle (1,2);
\fill[fill=red!8!white] (1,2) rectangle (2,3);
\fill[fill=red!8!white] (2,3) rectangle (3,5);
\fill[fill=red!8!white] (3,4) rectangle (4,6);
\fill[fill=red!8!white] (4,5) rectangle (5,7);
\fill[fill=red!8!white] (5,6) rectangle (6,7);
\draw (0,0)--(0,7)--(7,7);
\foreach \i in {1,...,6}
\draw (0,\i)--(\i,\i);
\foreach \i in {1,...,6}
\draw (\i,7)--(\i,\i);
\node (1) at (.5, .5) {$7$};
\node (2) at (1.5, 1.5) {$6$};
\node (3) at (2.5, 2.5) {$5$};
\node (4) at (3.5, 3.5) {$4$};
\node (5) at (4.5, 4.5) {$3$};
\node (6) at (5.5, 5.5) {$2$};
\node (7) at (6.5, 6.5) {$1$};
\node (8) at (.5, 6.5) {$\mathsf{X}$};
\node (8) at (1.5, 6.5) {$\mathsf{X}$};
\node (8) at (2.5, 6.5) {$\mathsf{X}$};
\node (8) at (3.5, 6.5) {$\mathsf{X}$};
\node (8) at (.5, 5.5) {$\mathsf{X}$};
\node (8) at (1.5, 5.5) {$\mathsf{X}$};
\node (8) at (2.5, 5.5) {$\mathsf{X}$};
\node (8) at (.5, 4.5) {$\mathsf{X}$};
\node (8) at (1.5, 4.5) {$\mathsf{X}$};
\node (8) at (.5, 3.5) {$\mathsf{X}$};
\node (8) at (1.5, 3.5) {$\mathsf{X}$};
\node (8) at (.5, 2.5) {$\mathsf{X}$};
\draw[very thick, color=dredcolor] (0,1)--(0,2)--(1,2)--(1,3)--(2,3)--(2,5)--(3,5)--(3,6)--(4,6)--(4,7)--(6,7)--(6,6)--(5,6)--(5,5)--(4,5)--(4,4)--(3,4)--(3,3)--(2,3)--(2,2)--(1,2)--(1,1)--(0,1)--cycle;
\node (9) at (5,2) {$\pi_{\bdnu}$};
\node (11) at (4.5, 6.5) {\color{blue}$13$};
\node (11) at (5.5, 6.5) {\color{blue}$12$};
\node (11) at (3.5, 5.5) {\color{blue}$24$};
\node (11) at (4.5, 5.5) {\color{blue}$23$};
\node (11) at (2.5, 4.5) {\color{blue}$35$};
\node (11) at (3.5, 4.5) {\color{blue}$34$};
\node (11) at (2.5, 3.5) {\color{blue}$45$};
\node (11) at (1.5, 2.5) {\color{blue}$56$};
\node (11) at (.5, 1.5) {\color{blue}$67$};
\node (9) at (8, 3) {$\Longleftrightarrow$};
\end{tikzpicture}
\qquad
\begin{tikzpicture}[scale=.4]
\filldraw [black] (0,0) circle (5pt)
                  (4,0) circle (5pt)
                  (0,4) circle (5pt)
                  (4,4) circle (5pt)
                  (7,2) circle (5pt)
                  (11,2) circle (5pt)
                  (15,2) circle (5pt);
\draw[thick] (0,4)--(0,0)--(4,0)--(4,4)--(0,0);
\draw[thick] (0,0)--(4,4);
\draw[thick] (0,4)--(4,0);
\draw[thick] (4,0)--(7,2);
\draw[thick] (4,4)--(7,2)--(15,2);
\node (1) at (0,-1) {$2$};
\node (1) at (0,5) {$1$};
\node (1) at (4,5) {$4$};
\node (1) at (4,-1) {$3$};
\node (1) at (7,1) {$5$};
\node (1) at (11,1) {$6$};
\node (1) at (15,1) {$7$};
\node (2) at (9, -2) {$G_{\bdnu}$};
\end{tikzpicture}
\caption{}\label{fig:example_G}
\end{figure}

\end{example}

We remark that given a coloring $\kappa$ of $G_{\bdnu}$, the statistic $\rm{asc}(\kappa)$ is consistent with the 
inversion statistic $\inv (\boldsymbol{S})$ obtained from the Dyck diagram $\pi_{\bdnu}$, where $\boldsymbol{S}$ is a word of 
length $n$, written in the main diagonal of $\pi_{\bdnu}$ from the south-west cell to the north-east cell (as we did in Example \ref{ex:dg}).
In fact, every unicellular LLT polynomial can be written as 
$$\LLT_{G}({\bf x};q) =\sum_{\kappa :V(G)\rightarrow \mathbb{Z}_{>0}} q^{\rm{asc}(\kappa)}x^{\kappa}$$
for some $G$ which is the incomparability graph of a natural unit interval order. By comparing to the definition of the chromatic quasisymmetric function $X_G ({\bf x};q)$, we can observe that the only difference is the \emph{proper} condition on the coloring $\kappa$. 
The precise relationship via \emph{plethysm} is given in \cite[Proposition 3.4]{CM}. 
\begin{prop}\cite[Proposition 3.4]{CM}
Let $G$ be the incomparability graph of a natural unit interval order with $n$ elements. Then we have 
$$ X_G ({\bf x};q)= (q-1)^{-n}\LLT_G [(q-1)X;q].$$

\end{prop}

Note that the square bracket on the right hand side of the above equation implies the \emph{plethystic substitution} and it is a convention that 
$X=x_1 + x_2 +\cdots$ in the plethystic substitution.
For the detailed explanation of it, we refer the reader to \cite{Hai2001, Hag08}.

Due to this relation between LLT polynomials and the chromatic quasisymmetric functions, we have the following equivalency of linear relations.

\begin{prop}\cite[Proposition 55]{AP}\label{prop:AP} Let $G_0,G_1,\dots,G_{\ell}$ be the incomparability graphs of natural unit interval orders. Then 
$$\sum_{i=0}^{\ell}c_i(q)X_{G_i}({\bf x};q)=0 \qquad\text{if and only if}\qquad \sum_{i=0}^{\ell}c_i(q)\LLT_{G_i}({\bf x};q)=0,$$ for some $c_i(q)$. \end{prop}


\bigskip


\section{The $k$-deletion property} \label{sec:local}


Recently, Lee \cite{Lee} provided a local linear relation between some unicellular LLT polynomials.

\begin{thm}\cite[Theorem 3.4]{Lee}(Local linear relation)\label{thm:Lee}
For an area sequence $a=(a_1,a_2,\dots,a_n)$ and $i$ such that $a_{i-1}+1\leq a_{i}$ (we set $a_0=1$), 
let $a^{0}=a,a^{1},a^{2}$ be area sequences defined by $a^{z}_j=a_i$ if $j\neq i$ and $a^{z}_i=a_i-z$ for $z=0,1,2$. If $a_{i+a_i-1}=a_{i+a_i}+1$, then 
$$\LLT_{\bdnu^0}({\bf x};q)+q\LLT_{\bdnu^2}({\bf x};q)=(1+q)\LLT_{\bdnu^1}({\bf x};q),$$
where $\bdnu^{z}$ is a unicellular LLT diagram satisfying that $a_{\bdnu^{z}}=a^{z}$ for $z=0,1,2$.

Equivalently, if we let $G_{z}$ be the graph $\bdnu^{z}_G$ for $z=0,1,2$, then   
$$X_{G_0}({\bf x};q)+q X_{G_2}({\bf x};q)=(1+q) X_{G_1}({\bf x};q).$$
\end{thm}
\smallskip
\begin{example}
Let $a=(2,3,3,2,1,1,0)$ be an area sequence and let $i=2$. Since $a_1+1\leq a_2$, $a^1=(2,2,3,2,1,1,0)$ and $a^2=(2,1,3,2,1,1,0)$ are well-defined. We note that $a_{2+a_2-1}=a_4=2$ and  $a_{2+a_2}=a_5=1$ so that $a_{2+a_2-1}=a_{2+a_2}+1$. Therefore, 
$$\LLT_{\bdnu^0}({\bf x};q)+q\LLT_{\bdnu^2}({\bf x};q)=(1+q)\LLT_{\bdnu^1}({\bf x};q),$$
where $\bdnu^{0}$, $\bdnu^{1}$, $\bdnu^{2}$ are defined as Figure \ref{fig:localex}.

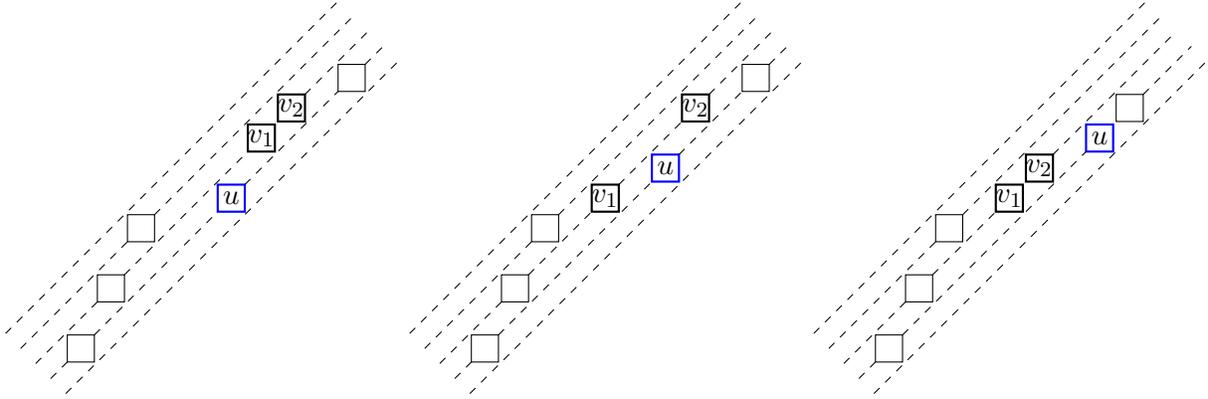
\begin{figure}[ht]
\begin{tikzpicture}[scale=.4]
\draw (1.05,-0.95)--(1.95,-0.95)--(1.95,-0.05)--(1.05,-0.05)--cycle;
\draw (2.05,1.05)--(2.95,1.05)--(2.95,1.95)--(2.05,1.95)--cycle;
\draw[thick, color=blue] (6.05,4.05)--(6.95,4.05)--(6.95,4.95)--(6.05,4.95)--cycle;
\draw (3.05,3.05)--(3.95,3.05)--(3.95,3.95)--(3.05,3.95)--cycle;
\draw[thick] (8.05,7.05)--(8.95,7.05)--(8.95,7.95)--(8.05,7.95)--cycle;
\draw[thick] (7.05,6.05)--(7.95,6.05)--(7.95,6.95)--(7.05,6.95)--cycle;
\draw (10.05,8.05)--(10.95,8.05)--(10.95,8.95)--(10.05,8.95)--cycle;

\draw[dashed] (-1,0)--(10.05,11.05);
\draw[dashed] (-.5,-.5)--(3.05,3.05)
              (3.95,3.95)--(10.55,10.55);
\draw[dashed] (0,-1)--(2.05,1.05)
              (2.95,1.95)--(7.05,6.05)
             (8.95,7.95)--(11.05,10.05);
\draw[dashed] (.5,-1.5)--(1.05,-.95)
              (1.95,-0.05)--(6.05,4.05)
             (6.95,4.95)--(10.05,8.05)
              (10.95,8.95)--(11.55,9.55);
\draw[dashed] (1,-2)--(12.05,9.05);
\node[] at (6.5,4.5) {$u$};
\node[] at (7.5, 6.5) {$v_1$};
\node[] at (8.5, 7.5) {$v_2$};
\end{tikzpicture}
\begin{tikzpicture}[scale=.4]
\draw (1.05,-0.95)--(1.95,-0.95)--(1.95,-0.05)--(1.05,-0.05)--cycle;
\draw (2.05,1.05)--(2.95,1.05)--(2.95,1.95)--(2.05,1.95)--cycle;
\draw[thick, color=blue] (7.05,5.05)--(7.95,5.05)--(7.95,5.95)--(7.05,5.95)--cycle;
\draw (3.05,3.05)--(3.95,3.05)--(3.95,3.95)--(3.05,3.95)--cycle;
\draw[thick] (5.05,4.05)--(5.95,4.05)--(5.95,4.95)--(5.05,4.95)--cycle;
\draw[thick] (8.05,7.05)--(8.95,7.05)--(8.95,7.95)--(8.05,7.95)--cycle;
\draw (10.05,8.05)--(10.95,8.05)--(10.95,8.95)--(10.05,8.95)--cycle;

\draw[dashed] (-1,0)--(10.05,11.05);
\draw[dashed] (-.5,-.5)--(3.05,3.05)
              (3.95,3.95)--(10.55,10.55);
\draw[dashed] (0,-1)--(2.05,1.05)
              (2.95,1.95)--(5.05,4.05)
              (5.95,4.95)--(8.05,7.05)
              (8.95,7.95)--(11.05,10.05);
\draw[dashed] (.5,-1.5)--(1.05,-.95)
              (1.95,-0.05)--(7.05,5.05)
              (7.95,5.95)--(10.05,8.05)
              (10.95,8.95)--(11.55,9.55);
\draw[dashed] (1,-2)--(12.05,9.05);
\node[] at (7.5,5.5) {$u$};
\node[] at (5.5, 4.5) {$v_1$};
\node[] at (8.5, 7.5) {$v_2$};
\end{tikzpicture}
\begin{tikzpicture}[scale=.4]
\draw (1.05,-0.95)--(1.95,-0.95)--(1.95,-0.05)--(1.05,-0.05)--cycle;
\draw (2.05,1.05)--(2.95,1.05)--(2.95,1.95)--(2.05,1.95)--cycle;
\draw[thick, color=blue] (8.05,6.05)--(8.95,6.05)--(8.95,6.95)--(8.05,6.95)--cycle;
\draw (3.05,3.05)--(3.95,3.05)--(3.95,3.95)--(3.05,3.95)--cycle;
\draw[thick] (5.05,4.05)--(5.95,4.05)--(5.95,4.95)--(5.05,4.95)--cycle;
\draw[thick] (6.05,5.05)--(6.95,5.05)--(6.95,5.95)--(6.05,5.95)--cycle;
\draw (9.05,7.05)--(9.95,7.05)--(9.95,7.95)--(9.05,7.95)--cycle;

\draw[dashed] (-1,0)--(10.05,11.05);
\draw[dashed] (-.5,-.5)--(3.05,3.05)
              (3.95,3.95)--(10.55,10.55);
\draw[dashed] (0,-1)--(2.05,1.05)
              (2.95,1.95)--(5.05,4.05)
              (6.95,5.95)--(11.05,10.05);
\draw[dashed] (.5,-1.5)--(1.05,-.95)
              (1.95,-0.05)--(8.05,6.05)
              (9.95,7.95)--(11.55,9.55);
\draw[dashed] (1,-2)--(12.05,9.05);
\node[] at (8.5,6.5) {$u$};
\node[] at (5.5, 4.5) {$v_1$};
\node[] at (6.5, 5.5) {$v_2$};
\end{tikzpicture}
\caption{$\bdnu^0$, $\bdnu^1$ and $\bdnu^2$ from the left.}\label{fig:localex}
\end{figure}

\end{example}

Remark that the condition $a_{i+a_i-1}=a_{i+a_i}+1$ prevents any cells existing on the left consecutive diagonal of the diagonal containing $v_1$ and $v_2$ in Figure \ref{fig:localex}, in between $v_1$ and $v_2$. If the condition is not satisfied, then the linear relation does not hold.
\begin{figure}[ht]
\begin{tikzpicture}[scale=.4]
\draw[thick] (0.05,0.05)--(0.95,0.05)--(0.95,0.95)--(0.05,0.95)--cycle;
\draw[thick] (1.05,2.05)--(1.95,2.05)--(1.95,2.95)--(1.05,2.95)--cycle;
\draw[thick, color=blue] (2.05,4.05)--(2.95,4.05)--(2.95,4.95)--(2.05,4.95)--cycle;
\draw[thick] (4.05,5.05)--(4.95,5.05)--(4.95,5.95)--(4.05,5.95)--cycle;
\draw[dashed] (-2,1)--(4.5,7.5);
\draw[dashed] (-1.5, .5)--(2.05, 4.05)
			(2.95,4.95)--(5,7);
\draw[dashed] (-1,0)--(1.05, 2.05)
			(1.95,2.95)--(4.05, 5.05)
			(4.95, 5.95)--(5.5, 6.5);
\draw[dashed] (-.5, -.5)--(.05, .05)
			(.95,.95)--(6,6)
			(0,-1)--(6.5,5.5);
\node[] at (.5, .5) {$u$};
\node[] at (1.5, 2.5) {$v_1$};
\node[] at (4.5, 5.5) {$v_2$};
\node[] at (2.5, 4.5) {$x$};
\end{tikzpicture}
\caption{}\label{fig:ctex}
\end{figure}
For example, the LLT diagram in Figure \ref{fig:ctex} has the area sequence $a= (2,1,1,0)$ which satisfies the condition $a_{i-1}+1\le a_i$ for $i=1$. 
However, due to the existence of the cell containing $x$, the condition $a_{i+a_i -1}=a_{i+a_i}+1$ is not satisfied for $i=1$; $a_{1+a_1 -1}= a_2 =1$ and $a_{1+a_1}+1=a_{3}+1=2$.
Thus, the linear relation does not hold among the LLT polynomials corresponding to the LLT diagrams $\bdnu^0$, $\bdnu^1$ and $\bdnu^2$, where $\bdnu^0$ is given in Figure \ref{fig:ctex} and $\bdnu^1$ and $\bdnu^2$ are obtained by moving the cell $u$ upward along the diagonal so that $v_1$ cell, and both of $v_1 $ and $v_2$ cells are 
on the left-below of the cell $u$, respectively.

We note that for the incomparability graphs of natural unit interval orders with some restrictions, Theorem \ref{thm:Lee} is a refinement of Triple-deletion property:

\begin{prop}\cite[Theorem 3.1]{OS}(Triple-deletion property) \label{prop:OS}Let $G$ be a graph with edge set $E$ such that $e_1,e_2,e_3\in E$ form a triangle. Then
$$X_{G}({\bf x})=X_{G-\{e_1\}}({\bf x})+X_{G-\{e_2\}}({\bf x})-X_{G-\{e_1,e_2\}}({\bf x}).$$  
\end{prop} 

We can further generalize the linear relations by applying Theorem \ref{thm:Lee} iteratively. 

\begin{thm}\label{thm:local} 
For an area sequence $a=(a_1,a_2,\dots,a_n)$, $2\le \ell\le n-1 $, and $i$ with $a_{i-1}+\ell-1\leq a_{i}$, 
let $a^{0}=a,a^{1},\dots,a^{\ell}$ be area sequences defined by $a^{z}_j=a_i$ if $j\neq i$ and $a^{z}_i=a_i-z$ for $z=0,1,\dots,\ell$. 
If 
$$a_{i+a_i-1}=a_{i+a_i}+1,\quad a_{i+a_i-2}=a_{i+a_i-1}+1,\quad \dots,\quad a_{i+a_i-\ell+1}=a_{i+a_i-\ell+2}+1,$$ 
then, for $1\le k\le \ell -1$, 

\begin{enumerate} 
\item[(a)] $\LLT_{\bdnu^0}({\bf x};q)+q[k]_q \LLT_{\bdnu^{k+1}}({\bf x};q)=[k+1]_q \LLT_{\bdnu^{k}}({\bf x};q)$,\\
\item[(b)] $[\ell-k]_q \LLT_{\bdnu^0}({\bf x};q)+q^{\ell-k}[k]_q \LLT_{\bdnu^{\ell}}({\bf x};q)=[\ell]_q \LLT_{\bdnu^k}({\bf x};q),$
\end{enumerate} 
\medskip
where $\bdnu^{z}$ is a unicellular LLT diagram satisfying that $a_{\bdnu^{z}}=a^{z}$ for $z=0,1,\dots,\ell$.

Equivalently, for $z=0,1,\dots,\ell$, if we simply denote $G_{\bdnu^{z}}$ by $G_{z}$, then for $1\le k\le \ell -1$,
\medskip
\begin{enumerate} 
\item[(a$'$)] $X_{G_0}({\bf x};q)+q[k]_q X_{G_{k+1}}({\bf x};q)=[k+1]_q X_{G_{k}}({\bf x};q)$,\\
\item[(b$'$)] $[\ell-k]_q X_{G_0}({\bf x};q)+q^{\ell-k}[k]_q X_{G_{\ell}}({\bf x};q)=[\ell]_q X_{G_k}({\bf x};q).$
\end{enumerate} 
\end{thm}
\smallskip
\begin{proof}
We note that for each $z\in \{0,1,\dots,\ell-2\}$, the area sequence $a^z$ satisfies the condition of Theorem \ref{thm:Lee}. Therefore, we have 
\begin{eqnarray*}
X_{G_0}({\bf x};q)+q X_{G_2}({\bf x};q)&=&[2]_q X_{G_1}({\bf x};q),\\
X_{G_1}({\bf x};q)+q X_{G_3}({\bf x};q)&=&[2]_q X_{G_2}({\bf x};q),\\
&\vdots&\\
X_{G_{\ell-2}}({\bf x};q)+q X_{G_{\ell}}({\bf x};q)&=&[2]_q X_{G_{\ell-1}}({\bf x};q).
\end{eqnarray*}

We prove the theorem by mathematical induction.
\begin{enumerate}
\item[(a$'$)] The case when $k=1$ is equivalent to Theorem \ref{thm:Lee}. So we may assume that this statement holds for $1\leq k \leq \ell-2$; 
$$X_{G_0}({\bf x};q)+q[k]_q X_{G_{k+1}}({\bf x};q)=[k+1]_q X_{G_k}({\bf x};q).$$
From Theorem \ref{thm:Lee}, it follows that
$$X_{G_0}({\bf x};q)+q[k]_q X_{G_{k+1}}({\bf x};q)=[k+1]_q([2]_q X_{G_{k+1}}({\bf x};q)-q X_{G_{k+2}}({\bf x};q)).$$
Then we obtain
\begin{eqnarray*}
X_{G_0}({\bf x};q)+q[k+1]_q X_{G_{k+2}}({\bf x};q)&=&([2]_q[k+1]_q-q[k]_q)X_{G_{k+1}}({\bf x};q)\\
&=&[k+2]_q X_{G_{k+1}}({\bf x};q),
\end{eqnarray*}
which complete the induction step and the proof.
\item[(b$'$)] The equation is equivalent to the following equation. 
$$[m]_q X_{G_0}({\bf x};q)+q^{m}[\ell-m]_q X_{G_{\ell}}({\bf x};q)=[\ell]_q X_{G_{\ell-m}}({\bf x};q).$$
If $m=1$, then it is true by (a$'$). Now we assume that
$$[m]_q X_{G_0}({\bf x};q)+q^{m}[\ell-m]_q X_{G_{\ell}}({\bf x};q)=[\ell]_q X_{G_{\ell-m}}({\bf x};q)$$
for $m<\ell-1$, then by (a$'$) we have
$$[m]_q X_{G_0}({\bf x};q)+q^{m}[\ell-m]_q X_{G_{\ell}}({\bf x};q)=\frac{[\ell]_q}{q[\ell-m-1]_q}([\ell-m]_q X_{G_{\ell-m-1}}({\bf x};q)-X_{G_0}({\bf x};q)),$$
or equivalently,
\begin{eqnarray*}
&&[\ell]_q[\ell-m]_q X_{G_{\ell-m-1}}({\bf x};q)\\ 
&& \qquad \qquad \qquad  =(q[\ell-m-1]_q[m]_q+[\ell]_q)X_{G_0}({\bf x};q)+q^{m+1}[\ell-m-1]_q[\ell-m]_q X_{G_{\ell}}({\bf x};q).
\end{eqnarray*}
Since $[\ell]_q=[m]_q+q^m[\ell-m]_q$,
$$q[\ell-m-1]_q[m]_q+[\ell]_q=q[\ell-m-1]_q[m]_q+([m]_q+q^m[\ell-m]_q)=[m+1]_q[\ell-m]_q.$$
From this, we have
$$[\ell]_q X_{G_{\ell-m-1}}({\bf x};q)=[m+1]_q X_{G_0}({\bf x};q)+q^{m+1}[\ell-m-1]_q X_{G_{\ell}}({\bf x};q),$$
as we desired. 
\end{enumerate}
\end{proof}


\section{$e$-expnasion of chromatic quasisymmetric functions related to certain graphs}\label{sec:lollipopG}


For two graphs $G$ and $H$ with vertex set $\{v_1,v_2,\dots,v_n\}$ and $\{v_n,v_{n+1},\dots,v_{n+m}\}$, respectively, let $G+H$ to be the graph with $V(G+H)=\{v_{1},v_{2},\dots,v_{n+m}\}$ and $E(G+H)=E(G)\cup E(H)$.\\
A graph $P_{n+1}+K_m$ is called a {\it lollipop graph} $L_{m,n}$ on $[m+n]$, where $P_{n+1}$ is a path on $[n+1]$ and $K_m$ is a complete graph with vertices $\{n+1,n+2,\dots,n+m\}$. We note that the lollipop graph $L_{m,n}$ on $[m+n]$ is the incomparability graph of the natural unit interval order $P(m_1,m_2,\dots,m_{m+n-1})$ such that  $m_i=i+1$ for $i\leq n$ and $m_{i}=n+m$ for $i> n$. Figure \ref{fig:lollipop} shows the lollipop graph $L_{6,5}$.

Recently, Dahlberg and van Willigenburg \cite{DW} gave an explicit $e$-positive formula for the chromatic symmetric function of a lollipop graph by iterating Triple-deletion property.

\begin{prop}\cite[Proposition 10]{DW} \label{prop:lollipop}For $m\geq 2$ and $n\geq 0$, $$X_{L_{m,n}}({\bf x})=\frac{(m-1)!}{(m+n-1)!}X_{K_{m+n}}({\bf x})+\sum_{i=0}^{n-1}\frac{(m+i-1)}{m(m+1)\cdots(m+i)} X_{P_{n-i}}({\bf x}) X_{K_{m+i}}({\bf x}).$$\end{prop}

In this section we give explicit $e$-positive and $e$-unimodal formulae for chromatic quasisymmetric functions of some graphs, generalizations of lollipop graphs, by using Theorem \ref{thm:local}. 


\subsection{Lollipop graphs}\label{sec:lollipop}

In this subsection we consider lollipop graphs. 

\begin{defn} For $m,n\geq0$, a \emph{lollipop quasisymmetric function} is given by $X_{L_{m,n}}({\bf x};q)$, that is, a chromatic quasisymmetric function of $L_{m,n}$. We simply denote $X_{L_{m,n}}({\bf x};1)$ by $X_{L_{m,n}}({\bf x})$.
\end{defn}

By the definition, $X_{L_{m,0}}({\bf x};q)=X_{K_m}({\bf x};q)$ and $X_{L_{0,n}}({\bf x};q)=X_{P_n}({\bf x};q)$, both of which are $e$-positive and $e$-unimodal, see  Proposition \ref{prop:basic}.\\ 

From Theorem \ref{thm:local} (a$'$), we have the following linear relation.

\begin{prop}
For integers $m\geq 2$ and $n\geq 0$,
\begin{equation}\label{eqn:lollipop} 
X_{L_{m,n}}({\bf x};q)=\frac{1}{[m]_q} \left( X_{L_{m+1,n-1}}({\bf x};q)+q[m-1]_q X_{P_n \cup K_m}({\bf x};q) \right).
\end{equation}
\end{prop}

\begin{proof}
We first note that the LLT diagram corresponding to the lollipop graph $L_{m+1,n-1}$ has the area sequence $a=(a_1,a_2,\dots,a_{n+m})$ such that $a_{i}=1$ for $i< n$ and $a_{i}=n+m-i$ for $i\geq n$. If we take $\ell=m$ and $i=n$, then the area sequence $a$ satisfies that the condition of Theorem \ref{thm:local}; 
\begin{itemize}
    \item $a_{i-1}+\ell-1=a_{n-1}+m-1=m\leq a_n$,
    \item $a_{i}=n+m-i=a_{i+1}+1$ \qquad for \quad $i\in\{n+1, n+2,\dots,n+m-1\}$.
\end{itemize}
Thus, by Theorem \ref{thm:local} (a$'$), we have
$$X_{L_{m+1,n-1}}({\bf x};q)+q[m-1]_q X_{G_{m}}({\bf x};q)=[m]_q X_{G_{m-1}}({\bf x};q),$$
where $G_{m-1}$ is the lollipop graph $L_{m,n}$ and $G_{m}$ is the graph $P_n \cup K_m$.
\end{proof}

From \eqref{eqn:lollipop}, we have a formula of a lollipop quasisymmetric function.

\begin{prop}\label{prop:quasilollipop}
For $m\geq 2$ and $n\geq 0$, a lollipop quasisymmetric function is
$$X_{L_{m,n}}({\bf x};q)=[m-1]_q!\left([m+n]_q \,e_{m+n}+\sum_{i=0}^{n-1}q[m+i-1]_q X_{P_{n-i}}({\bf x};q)\,e_{m+i}\right).$$
Consequently $X_{L_{m,n}}({\bf x};q)$ is $e$-positive and $e$-unimodal with center of symmetry $\frac{|E(L_{m,n})|}{2}$.
\end{prop}
\begin{proof}
If we use Equation (\ref{eqn:lollipop}) repeatedly, then we have a refinement of Proposition \ref{prop:lollipop};
$$X_{L_{m,n}}({\bf x};q)=\frac{X_{L_{m+k,n-k}}({\bf x};q)}{[m]_q\cdots[m+k-1]_q}+\sum_{i=0}^{k-1}\frac{q[m+i-1]_q}{[m]_q\cdots[m+i]_q} X_{P_{n-i}}({\bf x};q) X_{K_{m+i}}({\bf x};q).$$
If we let $k=n$, then $X_{L_{m+n,0}}({\bf x};q)=X_{k_{m+n}}({\bf x};q)$. Since $X_{K_{n}}({\bf x};q)=[n]_q! \,e_n$, 
\begin{eqnarray*}
X_{L_{m,n}}({\bf x};q)&=&\frac{[m+n]_q! \,e_{m+n}}{[m]_q\cdots[m+n]_q}+\sum_{i=0}^{n-1}\frac{q[m+i-1]_q}{[m]_q\cdots[m+i]_q}[m+i]_q!X_{P_{n-i}}({\bf x};q) \,e_{m+i}\\
&=&[m-1]_q!\left([m+n]_q \, e_{m+n}+\sum_{i=0}^{n-1}q[m+i-1]_q X_{P_{n-i}}({\bf x};q) \,e_{m+i}\right).
\end{eqnarray*}
By Proposition \ref{prop:basic} (b), $X_{P_{n-i}}({\bf x};q)$ is $e$-positive and $e$-unimodal with center of symmetry $\frac{n-i-1}{2}$. Since both of $[m+n]_q$ and $q[m+i-1]_q X_{P_{n-i}}({\bf x};q)$ are $e$-positive and $e$-unimodal with center of symmetry $\frac{n+m-1}{2}$, $X_{L_{m,n}}({\bf x};q)$ is $e$-positive and $e$-unimodal with center of symmetry $\frac{|E(L_{m,n})|}{2}$.
\end{proof}

In \cite{CH}, the first author and Cho obtained an $e$-positive and $e$-unimodal formula of the chromatic quasisymmetric function of $K_{r}+K_{n-r+1}$.

\begin{lem}\cite[Corollary 4.4]{CH}\label{lem:CH}
For $1\leq r \leq n-1$, let $G$ be the graph $K_{r}+K_{n-r+1}$. Then  
$$X_G({\bf x};q)=\sum_{i=0}^{\min \{n-r,r-1\}}q^{i}[n-r]_q![r-1]_q![n-2i]_q\, e_{(n-i,i)}.$$
\end{lem}

By combining Theorem \ref{thm:local} and Lemma \ref{lem:CH}, we obtain a formula of the chromatic quasisymmetric function of a graph $K_{r}+L_{m,n}$.
  
\begin{thm} \label{thm:double}
For $m\geq 3$, $1\leq r \leq m$, and $n\geq 0$, let $G$ be the graph $K_{r}+L_{m,n}$. If we let $d=n+m+r-1$, then
$$X_{G}({\bf x};q)=[m-1]_q!\left( \sum_{i=0}^{r-1}q^{i}[r-1]_q![d-2i]_q \,e_{(d-i,i)}  +\sum_{j=0}^{n-1}q[n+m-j-2]_q X_{L_{r,j}}({\bf x};q) \,e_{n+m-j-1}\right).$$
Hence, $X_{G}({\bf x};q)$ is $e$-positive and $e$-unimodal with center of symmetry $\frac{|E(G)|}{2}$.
\end{thm}

\begin{proof}
By Lemma \ref{lem:CH}, 
$$X_{K_{r}+L_{m+n,0}}({\bf x};q)=X_{K_{r}+K_{m+n}}({\bf x};q)=[n+m-1]_q!\sum_{i=0}^{r-1}q^{i}[r-1]_q![d-2i]_q \,e_{(d-i,i)}.$$
For convenience, we denote $X_{K_{r}+K_{m+n}}({\bf x};q)=[n+m-1]_q! \,f({\bf x};q)$.\\

From Theorem \ref{thm:local} (a$'$), we have the following relation for $0\leq k \leq n+m-3$,

\begin{equation}\label{eqn:double}
X_{K_{r}+L_{n+m-k-1,k+1}}({\bf x};q)=\frac{X_{K_{r}+L_{n+m-k,k}}({\bf x};q)+q[n+m-k-2]_q X_{L_{r,k}}({\bf x};q)X_{K_{n+m-k-1}}({\bf x};q)}{[n+m-k-1]_q}.
\end{equation}
\smallskip
If $k=0$, then \eqref{eqn:double} is equal to 
\begin{eqnarray*}
X_{K_{r}+L_{n+m-1,1}}({\bf x};q)&=&\frac{X_{K_{r}+L_{n+m,0}}({\bf x};q)+q[n+m-2]_q X_{L_{r,0}}({\bf x};q)X_{K_{n+m-1}}({\bf x};q)}{[n+m-1]_q}\\
&=&[n+m-2]_q!\left(f({\bf x};q)+q[n+m-2]_q X_{L_{r,0}}({\bf x};q) \,e_{n+m-1}\right).
\end{eqnarray*}
\smallskip
If $k=1$, then \eqref{eqn:double} is equal to 
\begin{eqnarray*}
X_{K_{r}+L_{n+m-2,2}}({\bf x};q)&=&\frac{X_{K_{r}+L_{n+m-1,1}}({\bf x};q)+q[n+m-3]_qX_{L_{r,1}}({\bf x};q)X_{K_{n+m-2}}({\bf x};q)}{[n+m-2]_q}\\
&=&[n+m-3]_q!\left(f({\bf x};q)+\sum_{j=0}^{1}q[n+m-j-2]_q X_{L_{r,j}}({\bf x};q)\,e_{n+m-j-1}\right).
\end{eqnarray*}
If we do this continually, then we have 
\begin{eqnarray*}
&&X_{K_{r}+L_{n+m-k-1,k+1}}({\bf x};q)\\
&&\hspace{30mm}=\frac{X_{K_{r}+L_{n+m-k,k}}({\bf x};q)+q[n+m-2]_q X_{L_{r,k}}({\bf x};q)X_{K_{n+m-k-1}}({\bf x};q)}{[n+m-k-1]_q}\\
&&\hspace{30mm}=[n+m-k-2]_q!\left(f({\bf x};q)+\sum_{j=0}^{k}q[n+m-j-2]_q X_{L_{r,j}}({\bf x};q) \,e_{n+m-j-1}\right),
\end{eqnarray*}
for $0\leq k\leq n-1$. In particular, when $k=n-1$ we get
\begin{eqnarray*}
X_{K_{r}+L_{m,n}}({\bf x};q)&=&\frac{X_{K_{r}+L_{m+1,n-1}}({\bf x};q)+q[m-1]_q X_{L_{r,n-1}}({\bf x};q)X_{K_{m}}({\bf x};q)}{[m]_q}\\
&=&[m-1]_q!\left(f({\bf x};q)+\sum_{j=0}^{n-1}q[n+m-j-2]_q X_{L_{r,j}}({\bf x};q) \,e_{n+m-j-1}\right),
\end{eqnarray*}
as we desired.
\end{proof}


\subsection{Melting lollipop graphs} \label{melting}


\begin{defn}
For integers $m,n\geq 0$, and $0\leq k \leq m-1$, a \emph{melting lollipop graph} $L^{(k)}_{m,n}$ on $[m+n]$ is obtained from the lollipop graph $L_{m,n}$ by deleting $k$ edges, 
$$\{n+1,n+m\}, ~\{n+1,n+m-1\},~ \dots,~ \{n+1,n+m-k+1\}.$$ 
We call $X_{L^{(k)}_{m,n}}({\bf x};q)$ the \emph{melting lollipop quasisymmetric function}. 
\end{defn}

A melting lollipop graph $L^{(k)}_{m,n}$ is the incomparability graph of a natural unit interval order $P(m_1,m_2,\dots,m_{m+n-1})$ with $m_i=i+1$ for $i\leq n$, $m_{n+1}=m+n-k$, and $m_{i}=m+n$ for $i>n+1$. Figure \ref{fig:melting} shows the melting lollipop graph $L^{(2)}_{6,5}$, for example.

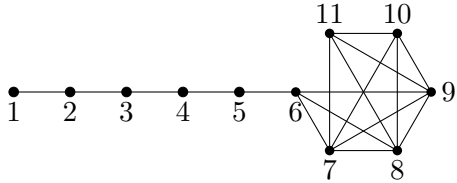
\begin{figure}[ht]
\begin{tikzpicture}[scale=.15]
\foreach \n in {6}
 \foreach \i in {1,...,\n}
    \fill (\i*360/\n:\n) coordinate (n\i) circle(2*\n pt)
      \ifnum \i>3 foreach \j in{\i,...,1}{(n\i)edge(n\j)}\fi;
\draw (180:6)--(-31,0);
\draw (60:6) -- (120:6);

\foreach \i in {1,...,5}
 \node[below] at (-36+5*\i,0) {\i}; 

\node[below] at (180:6) {6};
\node[below] at (240:6) {7};
\node[below] at (300:6) {8};
\node[right] at (0:6) {9};
\node[above] at (60:6) {10};
\node[above] at (120:6) {11};

\filldraw [black] (-11,0) circle (12pt)
			(-16,0) circle (12pt)
			(-21,0) circle (12pt)
			(-26,0) circle (12pt)
			(-31,0) circle (12pt);
\end{tikzpicture}
\caption{The incomparability graph $L^{(2)}_{6,5}$ of $P=P(2,3,4,5,6,9,11,11,11,11)$.}\label{fig:melting}
\end{figure}

By the definition, one can easily see that $L^{(0)}_{m,n}=L_{m,n}$, $L^{(m-2)}_{m,n}=L_{m-1,n+1}$, and $L^{(m-1)}_{m,n}$ is the disjoint union of $P_{n+1}$ and $K_{m-1}$. Therefore, from Theorem \ref{thm:local} (b$'$), we have the following relation,
\begin{equation}\label{eqn:mlollipoplr}
[m-k-1]_q X_{L_{m,n}}({\bf x};q)+q^{m-k-1}[k]_q X_{P_{n+1}\cup K_{m-1}}({\bf x};q)=[m-1]_q X_{L^{(k)}_{m,n}}({\bf x};q),
\end{equation}
which is equivalent to the following proposition.

\begin{prop}
\begin{equation}\label{eqn:melting}
X_{L^{(k)}_{m,n}}({\bf x};q)=\frac{[m-k-1]_q}{[m-1]_q}X_{L_{m,n}}({\bf x};q)+q^{m-k-1}[k]_q[m-2]_q!X_{P_{n+1}}({\bf x};q) \,e_{m-1}.
\end{equation}
\end{prop}

\smallskip
From \eqref{eqn:melting} and Proposition \ref{prop:quasilollipop}, we have formulae of melting lollipop quasisymmetric functions.

\begin{thm}\label{thm:melting}
For integers $m\geq2$, $n\geq0$, and $0\leq k\leq m-1$, a melting lollipop quasisymmetric function is
$$X_{L^{(k)}_{m,n}}({\bf x};q)=[m-k-1]_q [m-2]_q!\left( [m+n]_q \,e_{m+n} +\sum_{i=0}^{n-1}q[m+i-1]_q X_{P_{n-i}}({\bf x};q)\,e_{m+i}  \right)$$ 
\hspace{97mm}$+q^{m-k-1}[k]_q [m-2]_q!  X_{P_{n+1}}({\bf x};q)\,e_{m-1}$.\\[2ex]
Consequently $X^{(k)}_{L_{m,n}}({\bf x};q)$ is $e$-positive and $e$-unimodal with center of symmetry $\frac{|E(L^{(k)}_{m,n})|}{2}$.
\end{thm}


\section{Schur expansion of unicellular LLT polynomials related to certain graphs}\label{sec:LLT_Schur}


Due to the equivalency given in Proposition \ref{prop:AP} between LLT polynomials and chromatic quasisymmetric functions, the relations satisfied by 
$X_G (\bx ;q)$ given in Section \ref{sec:lollipopG} can be restated in terms of LLT polynomials. In this section, we utilize those relations to prove 
Schur expansion formulas of LLT polynomials corresponding to certain graphs. 
We first define a statistic which will be used in the description of Schur coefficients of LLT polynomials.

Given a partition $\lambda$, we define the \emph{reading order} as the total ordering of the cells in $\lambda$ by reading them row by row, from top 
to bottom, and from left to right within each row. Then for a standard Young tableau $T\in \SYT(\lambda)$, the \emph{descent set} $D(T)$ is defined by 
$$D(T)=\{i~:~ T^{-1}(i+1) \text{ precedes } T^{-1}(i) \text{ in the reading order}\}.$$

\begin{defn}\label{def:wt}
Given a unicellular LLT diagram $\bdnu = (\nu^{(1)},\dots, \nu^{(n)})$, an $n$-tuple of single cells, say the corresponding Dyck diagram $\pi_{\bdnu}$
has the area sequence $a=(a_1,\dots a_{n-1},0)$. Then for any partition $\lambda\vdash n$ and a standard Young tableau $T\in\SYT(\lambda)$, define 
$$wt_{\bdnu}(T) = \sum_{i \in D(T)}a_i.$$
\end{defn}


Our purpose is
to give a combinatorial interpretation of Schur coefficients appearing in the Schur expansion of unicellular LLT polynomials 
in terms of $q$ polynomials weighted by the statistic in Definition \ref{def:wt}.
More precisely,
we will introduce certain classes of unicellular LLT diagrams  $\bdnu = (\nu^{(1)},\dots, \nu^{(n)})$ satisfying that 
\[
\LLT_{\bdnu} ({\bf x};q) =  \sum_{\lambda \, \vdash n} \left( \sum_{P \in \SYT(\lambda)} q^{wt_{\bdnu}(P)} \right) s_\lambda \, .
\]

To begin with,
we introduce the following lemma which will play a key role in proving the Schur coefficients of LLT polynomials in the rest of this section. 

\begin{lem}\label{lem:prod_of_LLT}
Let $G_1$ and $G_2$ be graphs of order $n$ and $m$  associated to the LLT diagram
$\bdnu_1$ and $\bdnu_2$, respectively.
Let $G$ be the graph $G_1 \cup G_2$ of order $(n+m)$ and
let $\bdnu$ be the LLT diagram corresponding to $G$.
If
\[
\LLT_{\bdnu_1}({\bf x};q) = \sum_{\lambda \, \vdash n} \left( \sum_{P \in \SYT(\lambda)} q^{wt_{\bdnu_1}(P)} \right) s_\lambda
\quad \text{and} \quad
\LLT_{\bdnu_2}({\bf x};q) = \sum_{\lambda \, \vdash m} \left( \sum_{Q \in \SYT(\lambda)} q^{wt_{\bdnu_2}(Q)} \right) s_\lambda \, ,
\]
then
\[
\LLT_{\bdnu}({\bf x};q) = \sum_{\lambda \, \vdash (n+m)} \left( \sum_{T \in \SYT(\lambda)} q^{wt_{\bdnu}(T)} \right) s_\lambda \, .
\]
\end{lem}

To prove the above lemma,
we briefly introduce the switching algorithms on standard Young tableaux 
introduced by Benkart, Sottile, and Stroomer \cite{BSS},
which is built upon Sch\"{u}zenberge's jeu de taquin sliding process.
For partitions $\lambda \subseteq \mu \subseteq \nu$,
let $T$ and $S$ be standard Young tableaux of shape $\mu/\lambda$ and  $\nu/\mu$, respectively.
The tableau switching on $(T,S)$ is a combianatorial algorithm 
to apply jeu de taquin slides to $S$ iteratively 
following the order in which we choose empty boxes given by $T$
from the largest entry to the smallest entry.
For example, if
\begin{displaymath}
\begin{tikzpicture}[scale=0.8]
\def \hhh{5mm}    
\def \vvv{5mm}    
\def \hhhhh{70mm}  
\node[] at (\hhh*0.8,-\vvv*0.4) {$T=$};
\draw[-,black!10] (\hhh*3,-\vvv*2) rectangle (\hhh*4,-\vvv*1);
\draw[fill=black!30] (\hhh*3,0) rectangle (\hhh*4,\vvv*1);
%
\draw[fill=black!30] (\hhh*3,-\vvv*1) rectangle (\hhh*4,\vvv*0);
\draw[fill=black!30] (\hhh*4,-\vvv*1) rectangle (\hhh*5,\vvv*0);
%
\draw[fill=black!30] (\hhh*4,-\vvv*2) rectangle (\hhh*5,-\vvv*1);
\node[] at (\hhh*3.5,\vvv*0.5) {$4$};
\node[] at (\hhh*3.5,-\vvv*0.5) {$2$};
\node[] at (\hhh*4.5,-\vvv*0.5) {$3$};
\node[] at (\hhh*4.5,-\vvv*1.5) {$1$};
\node[] at (\hhh*3.5+\hhhhh*0.5,-\vvv*0.4) {and};
\node[] at (\hhh*0.8+\hhhhh,-\vvv*0.4) {$S=$};
\draw[-,black!10] (\hhh*3+\hhhhh,0) rectangle (\hhh*4+\hhhhh,\vvv*1);
%
\draw[-,black!10] (\hhh*3+\hhhhh,-\vvv*1) rectangle (\hhh*4+\hhhhh,\vvv*0);
\draw[-,black!10] (\hhh*4+\hhhhh,-\vvv*1) rectangle (\hhh*5+\hhhhh,\vvv*0);
\draw[-] (\hhh*5+\hhhhh,-\vvv*1) rectangle (\hhh*6+\hhhhh,\vvv*0);
\draw[-] (\hhh*4+\hhhhh,0) rectangle (\hhh*5+\hhhhh,\vvv*1);
\draw[-] (\hhh*5+\hhhhh,0) rectangle (\hhh*6+\hhhhh,\vvv*1);
\draw[-,black!10] (\hhh*3+\hhhhh,-\vvv*2) rectangle (\hhh*4+\hhhhh,-\vvv*1);
\draw[-,black!10] (\hhh*4+\hhhhh,-\vvv*2) rectangle (\hhh*5+\hhhhh,-\vvv*1);
\draw[-] (\hhh*5+\hhhhh,-\vvv*2) rectangle (\hhh*6+\hhhhh,-\vvv*1);
\draw[-] (\hhh*6+\hhhhh,-\vvv*2) rectangle (\hhh*7+\hhhhh,-\vvv*1);
%
\node[] at (\hhh*4.5+\hhhhh,\vvv*0.5) {$4$};
\node[] at (\hhh*5.5+\hhhhh,\vvv*0.5) {$5$};
\node[] at (\hhh*5.5+\hhhhh,-\vvv*0.5) {$2$};
\node[] at (\hhh*5.5+\hhhhh,-\vvv*1.5) {$1$};
\node[] at (\hhh*6.5+\hhhhh,-\vvv*1.5) {$3$};
\end{tikzpicture} \, ,
\end{displaymath}
then
the followings illustrate how the tableau switching algorithm acts on $(T,S)$ step by step:
\begin{displaymath}
\begin{tikzpicture}[scale=0.8]
\def \hhh{5mm}    
\def \vvv{5mm}    
\draw[-,black!10] (\hhh*3,-\vvv*2) rectangle (\hhh*4,-\vvv*1);
\draw[fill=black!30] (\hhh*3,0) rectangle (\hhh*4,\vvv*1);
\draw[-] (\hhh*4,0) rectangle (\hhh*5,\vvv*1);
\draw[-] (\hhh*5,0) rectangle (\hhh*6,\vvv*1);
%
\draw[fill=black!30] (\hhh*3,-\vvv*1) rectangle (\hhh*4,\vvv*0);
\draw[fill=black!30] (\hhh*4,-\vvv*1) rectangle (\hhh*5,\vvv*0);
\draw[-] (\hhh*5,-\vvv*1) rectangle (\hhh*6,\vvv*0);
%
\draw[fill=black!30] (\hhh*4,-\vvv*2) rectangle (\hhh*5,-\vvv*1);
\draw[-] (\hhh*5,-\vvv*2) rectangle (\hhh*6,-\vvv*1);
\draw[-] (\hhh*6,-\vvv*2) rectangle (\hhh*7,-\vvv*1);
%
\node[] at (\hhh*3.5,\vvv*0.5) {$4$};
\node[] at (\hhh*4.5,\vvv*0.5) {$4$};
\node[] at (\hhh*5.5,\vvv*0.5) {$5$};
\node[] at (\hhh*3.5,-\vvv*0.5) {$2$};
\node[] at (\hhh*4.5,-\vvv*0.5) {$3$};
\node[] at (\hhh*5.5,-\vvv*0.5) {$2$};
\node[] at (\hhh*4.5,-\vvv*1.5) {$1$};
\node[] at (\hhh*5.5,-\vvv*1.5) {$1$};
\node[] at (\hhh*6.5,-\vvv*1.5) {$3$};
%
\end{tikzpicture}
\begin{tikzpicture}[scale=0.8]
\def \hhh{5mm}    
\def \vvv{5mm}    
\draw[->,decorate,decoration={snake,amplitude=.4mm,segment length=2mm,post length=1mm}] (\hhh*0,-\vvv*0.5) -- (\hhh*1.7,-\vvv*0.5);
\draw[-,black!10] (\hhh*3,-\vvv*2) rectangle (\hhh*4,-\vvv*1);
\draw[-] (\hhh*3,0) rectangle (\hhh*4,\vvv*1);
\draw[-] (\hhh*4,0) rectangle (\hhh*5,\vvv*1);
\draw[fill=black!30] (\hhh*5,0) rectangle (\hhh*6,\vvv*1);
%
\draw[fill=black!30] (\hhh*3,-\vvv*1) rectangle (\hhh*4,\vvv*0);
\draw[fill=black!30] (\hhh*4,-\vvv*1) rectangle (\hhh*5,\vvv*0);
\draw[-] (\hhh*5,-\vvv*1) rectangle (\hhh*6,\vvv*0);
%
\draw[fill=black!30] (\hhh*4,-\vvv*2) rectangle (\hhh*5,-\vvv*1);
\draw[-] (\hhh*5,-\vvv*2) rectangle (\hhh*6,-\vvv*1);
\draw[-] (\hhh*6,-\vvv*2) rectangle (\hhh*7,-\vvv*1);
%
\node[] at (\hhh*3.5,\vvv*0.5) {$4$};
\node[] at (\hhh*4.5,\vvv*0.5) {$5$};
\node[] at (\hhh*5.5,\vvv*0.5) {$4$};
\node[] at (\hhh*3.5,-\vvv*0.5) {$2$};
\node[] at (\hhh*4.5,-\vvv*0.5) {$3$};
\node[] at (\hhh*5.5,-\vvv*0.5) {$2$};
\node[] at (\hhh*4.5,-\vvv*1.5) {$1$};
\node[] at (\hhh*5.5,-\vvv*1.5) {$1$};
\node[] at (\hhh*6.5,-\vvv*1.5) {$3$};
%
\end{tikzpicture}
\begin{tikzpicture}[scale=0.8]
\def \hhh{5mm}    
\def \vvv{5mm}    
\draw[->,decorate,decoration={snake,amplitude=.4mm,segment length=2mm,post length=1mm}] (\hhh*0,-\vvv*0.5) -- (\hhh*1.7,-\vvv*0.5);
%
\draw[-,black!10] (\hhh*3,-\vvv*2) rectangle (\hhh*4,-\vvv*1);
\draw[-] (\hhh*3,0) rectangle (\hhh*4,\vvv*1);
\draw[-] (\hhh*4,0) rectangle (\hhh*5,\vvv*1);
\draw[fill=black!30] (\hhh*5,0) rectangle (\hhh*6,\vvv*1);
%
\draw[fill=black!30] (\hhh*3,-\vvv*1) rectangle (\hhh*4,\vvv*0);
\draw[-] (\hhh*4,-\vvv*1) rectangle (\hhh*5,\vvv*0);
\draw[fill=black!30] (\hhh*5,-\vvv*1) rectangle (\hhh*6,\vvv*0);
%
\draw[fill=black!30] (\hhh*4,-\vvv*2) rectangle (\hhh*5,-\vvv*1);
\draw[-] (\hhh*5,-\vvv*2) rectangle (\hhh*6,-\vvv*1);
\draw[-] (\hhh*6,-\vvv*2) rectangle (\hhh*7,-\vvv*1);
%
\node[] at (\hhh*3.5,\vvv*0.5) {$4$};
\node[] at (\hhh*4.5,\vvv*0.5) {$5$};
\node[] at (\hhh*5.5,\vvv*0.5) {$4$};
\node[] at (\hhh*3.5,-\vvv*0.5) {$2$};
\node[] at (\hhh*4.5,-\vvv*0.5) {$2$};
\node[] at (\hhh*5.5,-\vvv*0.5) {$3$};
\node[] at (\hhh*4.5,-\vvv*1.5) {$1$};
\node[] at (\hhh*5.5,-\vvv*1.5) {$1$};
\node[] at (\hhh*6.5,-\vvv*1.5) {$3$};
%
%
\end{tikzpicture}
\begin{tikzpicture}[scale=0.8]
\def \hhh{5mm}    
\def \vvv{5mm}    
\draw[->,decorate,decoration={snake,amplitude=.4mm,segment length=2mm,post length=1mm}] (\hhh*0,-\vvv*0.5) -- (\hhh*1.7,-\vvv*0.5);
%
\draw[-,black!10] (\hhh*3,-\vvv*2) rectangle (\hhh*4,-\vvv*1);
\draw[-] (\hhh*3,0) rectangle (\hhh*4,\vvv*1);
\draw[fill=black!30] (\hhh*4,0) rectangle (\hhh*5,\vvv*1);
\draw[fill=black!30] (\hhh*5,0) rectangle (\hhh*6,\vvv*1);
%
\draw[-] (\hhh*3,-\vvv*1) rectangle (\hhh*4,\vvv*0);
\draw[-] (\hhh*4,-\vvv*1) rectangle (\hhh*5,\vvv*0);
\draw[fill=black!30] (\hhh*5,-\vvv*1) rectangle (\hhh*6,\vvv*0);
%
\draw[fill=black!30] (\hhh*4,-\vvv*2) rectangle (\hhh*5,-\vvv*1);
\draw[-] (\hhh*5,-\vvv*2) rectangle (\hhh*6,-\vvv*1);
\draw[-] (\hhh*6,-\vvv*2) rectangle (\hhh*7,-\vvv*1);
%
\node[] at (\hhh*3.5,\vvv*0.5) {$4$};
\node[] at (\hhh*4.5,\vvv*0.5) {$2$};
\node[] at (\hhh*5.5,\vvv*0.5) {$4$};
\node[] at (\hhh*3.5,-\vvv*0.5) {$2$};
\node[] at (\hhh*4.5,-\vvv*0.5) {$5$};
\node[] at (\hhh*5.5,-\vvv*0.5) {$3$};
\node[] at (\hhh*4.5,-\vvv*1.5) {$1$};
\node[] at (\hhh*5.5,-\vvv*1.5) {$1$};
\node[] at (\hhh*6.5,-\vvv*1.5) {$3$};
%
%
\end{tikzpicture}
\begin{tikzpicture}[scale=0.8]
\def \hhh{5mm}    
\def \vvv{5mm}    
\draw[->,decorate,decoration={snake,amplitude=.4mm,segment length=2mm,post length=1mm}] (\hhh*0,-\vvv*0.5) -- (\hhh*1.7,-\vvv*0.5) node[midway,below] {};
%
\draw[-,black!10] (\hhh*3,-\vvv*2) rectangle (\hhh*4,-\vvv*1);
\draw[-] (\hhh*3,0) rectangle (\hhh*4,\vvv*1);
\draw[fill=black!30] (\hhh*4,0) rectangle (\hhh*5,\vvv*1);
\draw[fill=black!30] (\hhh*5,0) rectangle (\hhh*6,\vvv*1);
%
\draw[-] (\hhh*3,-\vvv*1) rectangle (\hhh*4,\vvv*0);
\draw[-] (\hhh*4,-\vvv*1) rectangle (\hhh*5,\vvv*0);
\draw[fill=black!30] (\hhh*5,-\vvv*1) rectangle (\hhh*6,\vvv*0);
%
\draw[-] (\hhh*4,-\vvv*2) rectangle (\hhh*5,-\vvv*1);
\draw[-] (\hhh*5,-\vvv*2) rectangle (\hhh*6,-\vvv*1);
\draw[fill=black!30] (\hhh*6,-\vvv*2) rectangle (\hhh*7,-\vvv*1);
%
\node[] at (\hhh*3.5,\vvv*0.5) {$4$};
\node[] at (\hhh*4.5,\vvv*0.5) {$2$};
\node[] at (\hhh*5.5,\vvv*0.5) {$4$};
\node[] at (\hhh*3.5,-\vvv*0.5) {$2$};
\node[] at (\hhh*4.5,-\vvv*0.5) {$5$};
\node[] at (\hhh*5.5,-\vvv*0.5) {$3$};
\node[] at (\hhh*4.5,-\vvv*1.5) {$1$};
\node[] at (\hhh*5.5,-\vvv*1.5) {$3$};
\node[] at (\hhh*6.5,-\vvv*1.5) {$1$};
%
%
\end{tikzpicture}
\end{displaymath}
As we can see in the above example,
the tableau switching on $(T,S)$ results another pair of standard Young tableaux
and we denote it by $({}^TS,T_S)$ to respect notation in \cite{BSS}.
In the above case,
\begin{displaymath}
\begin{tikzpicture}[scale=0.8]
\def \hhh{5mm}    
\def \vvv{5mm}    
\def \hhhhh{70mm}  
\node[] at (\hhh*0.8,-\vvv*0.4) {${}^TS=$};
\draw[-,black!10] (\hhh*3,-\vvv*2) rectangle (\hhh*4,-\vvv*1);
\draw[-] (\hhh*3,0) rectangle (\hhh*4,\vvv*1);
%
\draw[-] (\hhh*3,-\vvv*1) rectangle (\hhh*4,\vvv*0);
\draw[-] (\hhh*4,-\vvv*1) rectangle (\hhh*5,\vvv*0);
%
\draw[-] (\hhh*4,-\vvv*2) rectangle (\hhh*5,-\vvv*1);
\draw[-] (\hhh*5,-\vvv*2) rectangle (\hhh*6,-\vvv*1);
\node[] at (\hhh*3.5,\vvv*0.5) {$4$};
%
\node[] at (\hhh*3.5,-\vvv*0.5) {$2$};
\node[] at (\hhh*4.5,-\vvv*0.5) {$5$};
%
\node[] at (\hhh*4.5,-\vvv*1.5) {$1$};
\node[] at (\hhh*5.5,-\vvv*1.5) {$3$};
\node[] at (\hhh*3.5+\hhhhh*0.5,-\vvv*0.4) {and};
\node[] at (\hhh*0.8+\hhhhh,-\vvv*0.4) {$T_S=$};
%
\draw[-,black!10] (\hhh*3+\hhhhh,0) rectangle (\hhh*4+\hhhhh,\vvv*1);
\draw[-,black!10] (\hhh*3+\hhhhh,-\vvv*1) rectangle (\hhh*4+\hhhhh,\vvv*0);
\draw[-,black!10] (\hhh*4+\hhhhh,-\vvv*1) rectangle (\hhh*5+\hhhhh,\vvv*0);
\draw[-,black!10] (\hhh*3+\hhhhh,-\vvv*2) rectangle (\hhh*4+\hhhhh,-\vvv*1);
\draw[-,black!10] (\hhh*4+\hhhhh,-\vvv*2) rectangle (\hhh*5+\hhhhh,-\vvv*1);
\draw[-,black!10] (\hhh*5+\hhhhh,-\vvv*2) rectangle (\hhh*6+\hhhhh,-\vvv*1);
\draw[fill=black!30] (\hhh*4+\hhhhh,0) rectangle (\hhh*5+\hhhhh,\vvv*1);
\draw[fill=black!30] (\hhh*5+\hhhhh,0) rectangle (\hhh*6+\hhhhh,\vvv*1);
\draw[fill=black!30] (\hhh*5+\hhhhh,-\vvv*1) rectangle (\hhh*6+\hhhhh,\vvv*0);
\draw[fill=black!30] (\hhh*6+\hhhhh,-\vvv*2) rectangle (\hhh*7+\hhhhh,-\vvv*1);
\node[] at (\hhh*4.5+\hhhhh,\vvv*0.5) {$2$};
\node[] at (\hhh*5.5+\hhhhh,\vvv*0.5) {$4$};
\node[] at (\hhh*5.5+\hhhhh,-\vvv*0.5) {$3$};
\node[] at (\hhh*6.5+\hhhhh,-\vvv*1.5) {$1$};
%
\end{tikzpicture} \, .
\end{displaymath}

\begin{lem}\label{lemma:switching}\cite{BSS}
  Let $T$ and $S$ be standard Young tableaux of shape  $\mu/\lambda$ and  $\nu/\mu$, respectively.
  Assume that the tableau switching on $(T,S)$ transforms $T$ into $T_S$ and $S$ into ${}^TS$. Then 
  \begin{enumerate}
  \item[(a)] $T$ and $T_S$ are Knuth equivalent, that is, they have the same rectification.
  \item[(b)] $S$ and ${}^TS$ are Knuth equivalent, that is, they have the same rectification.
  \item[(c)] The tableau switching ${}^TS$ and $T_S$ transforms ${}^TS$ into $S$ and $T_S$ into $T$.
  \end{enumerate}
\end{lem}

\smallskip
\noindent
\emph{Proof of Lemma \ref{lem:prod_of_LLT}.}
For simplicity,
we write
$$c_\lambda(q) = \sum_{T \in \SYT(\lambda)} q^{wt_{\bdnu_1}(T)}
\quad \text{and} \quad
d_\mu(q) = \sum_{T \in \SYT(\mu)} q^{wt_{\bdnu_2}(T)} \, . $$
Then we have
\[
\begin{aligned}
  \LLT_{\bdnu}({\bf x};q) &= \LLT_{\bdnu_{G_1 \cup G_2}}({\bf x};q) =\LLT_{\bdnu_1}({\bf x};q) \cdot \LLT_{\bdnu_2}({\bf x};q) \\
  &= \left( \sum_{\lambda \, \vdash n} c_{\lambda}(q) s_\lambda \right)
  \cdot \left( \sum_{\mu \, \vdash m} d_{\mu}(q) s_\mu \right) \\
  &= \sum_{\nu \, \vdash (n+m)}\left( \sum_{\substack{(\lambda,\mu) \\ {\lambda \, \vdash n} \\ 
  {\mu \, \vdash m}}}c_\lambda^{\nu/\mu} c_\lambda(q) d_\mu(q) \right) s_\nu \, .
\end{aligned}
\]
Let $\mathcal{C}^{\nu/\mu}_\lambda$ be the set of all standard Young tableaux of shape $\nu/\mu$
whose rectification is the row tableau $R_\lambda$.
To prove our assertion,
it is enough to show that for each $\nu \vdash (n+m)$
there is a correspondence
\[
\varphi :
\bigcup_{\substack{(\lambda,\mu) \\ {\lambda \vdash n} \\ {\mu \vdash m}}}
\left\{ (R,P,Q) \, : \, R \in \mathcal{C}^{\nu/\mu}_\lambda , P \in \SYT(\lambda) \text{ and } Q \in \SYT(\mu) \right\} \rightarrow \{ T \, : \, T \in \SYT(\nu) \}
\]
satisfying that
\begin{enumerate}
  \item $\varphi$ is bijective, and
  \item $\varphi$ is weight-preserving, that is, if $\varphi : (R,P,Q) \mapsto T $, then $wt_{\bdnu_1}(P) + wt_{\bdnu_2}(Q) = wt_{\bdnu}(T)$.
\end{enumerate}
Indeed,
we construct such a bijection by means of the tableau switching as follows:
note that the tableau switching on $(Q,R)$ results a pair $(R_\lambda, Q_R)$ of standard Young tableaux such that $Q_R$ is of shape $\nu/\lambda$.
Let $\hat{Q}_R$ be a filling obtained from $Q_R$ by replacing the entry $i$ with $n+i$
for each $i$.
Obviously,
$\hat{Q}_R$ is a standard Young tableau of shape $\nu/\lambda$
with entries from $\{ n+1, n+2, \cdots, n+m \}$.
Now we define $\varphi((R,P,Q))$ by $P \cup \hat{Q}_R$,
which is well defined because ${Q}_R$ is uniquely determined due to Lemma \ref{lemma:switching}(c).

On the other hand,
when $T \in \SYT(\nu)$
let us denote $T^{(n)}$ be a subtableau of $T$ consisting of $\{ 1, 2, \cdots , n \}$.
And let $T^{(m)}$ be a filling obtained from $T\setminus T^{(n)}$
by replacing the entry $j$ with $j-n$.
Then $T^{(n)}$ and $T^{(m)}$ is a standard Young tableau of shape $\lambda$ and $\nu/\lambda$ for some $\lambda \vdash n$, respectively.
Applying the tableau switching on $(R_\lambda, T^{(m)})$,
we get a pair $(\text{Rect}(T^{(m)}), R_{T^{(m)}})$ of standard Young tableaux,
where $\text{Rect}(T^{(m)})$ is the rectification of $T^{(m)}$.
Moreover, 
we know that $R_{T^{(m)}}$ is Knuth equivalent to $R_\lambda$
due to Lemma \ref{lemma:switching}(a).
Thus $\text{Rect}(T^{(m)}) \in \SYT(\mu)$ for some $\mu \vdash m$ and
$R_{T^{(m)}} \in \mathcal{C}_\lambda^{\nu/\mu}$.
In all, we can conclude that
for each $T \in \SYT(\nu)$ with $\nu \vdash (n+m)$,
the tableau switching produces the triple $(R_{T^{(m)}}, T^{(n)}, \text{Rect}(T^{(m)}))$ such that
$R_{T^{(m)}} \in \mathcal{C}_\lambda^{\nu/\mu}$, $T^{(n)} \in \SYT(\lambda)$ and $\text{Rect}(T^{(m)}) \in \SYT(\mu)$ for some $\lambda \vdash n$ and $\mu \vdash m$.
Furthermore,
it follows from the ivolutiveness of the tableau switching that
$\varphi\left((R_{T^{(m)}}, T^{(n)}, \text{Rect}(T^{(m)}) \right) = T$,
which shows that $\varphi$ is bijective.

In order to prove that
$\varphi$ is weight-preserving,
we recall the well known fact that the descent set of a given Young tableau is invariant
under applying forward or reverse jeu de taquin slides.
Hence, if $\varphi((R,P,Q))=T$, then 
$i \in D(P)$ (resp., $i-n \in D(Q)$) if and only if $i \in D(T)$ for $1 \leq i \leq n-1$ (resp., $n+1 \leq i \leq n+m-1$).
If we let the area sequences of $\bdnu_1$ and $\bdnu_2$ be $a_{\bdnu_1}= (\alpha_1, \alpha_2, \cdots, \alpha_{n-1},0)$ and $a_{\bdnu_2}= (\beta_1, \beta_2, \cdots, \beta_{m-1},0)$, respectively,
then the corresponding area sequence $a_{\bdnu} = (a_1, a_2, \cdots, a_{n+m})$ is of the form
\[
a_i =
\begin{cases}
  \alpha_i \quad \quad \text{if } 1\leq i \leq n-1 \\
   0 \quad  \quad \ \text{if } i=n \\
  \beta_{i-n} \quad \text{if } n+1 \leq i \leq n+m-1 \\
   0  \quad  \ \quad \text{if } i=n+m
\end{cases} \, .
\]
Therefore, $wt_{\bdnu_1}(P) + wt_{\bdnu_2}(Q)$ and $wt_{\bdnu}(T)$ are the same.
It should be noted that 
$n$ might be a descent of $T$. 
But it dose not matter to our assertion since $a_n$ is always $0$.
\qed

\begin{rem}
\rm{The properties of the tableau switching described in Lemma \ref{lemma:switching} played a key role in proving Lemma \ref{lem:prod_of_LLT}.
These properties are well extended to the case of $m$-fold multitableaux in \cite{BSS}, and hence one can prove the following:
let $G_i$ be graph of order $n_i$ associated to the LLT diagram $\bdnu_i$ for $1 \leq i \leq m$ and $G = \bigcup_i G_i$ graph of order $\sum_i n_i := n$.
If for every $1 \leq i \leq m$
$\LLT_{\bdnu_i} ({\bf x};q) = \sum_{\lambda \, \vdash n_i} \left( \sum_{P \in \SYT(\lambda)} q^{wt_{\bdnu_i}(P)} \right) s_\lambda$, then 
\[
\LLT_{\bdnu}({\bf x};q) = \sum_{\lambda \, \vdash n} \left( \sum_{T \in \SYT(\lambda)} q^{wt_{\bdnu}(T)} \right) s_\lambda \, 
\]
where $\bdnu$ is the LLT diagram corresponding to $G$.
This can be proved in a similar way as in the proof of Lemma \ref{lem:prod_of_LLT} 
so we omit the detailed proof.
}
\end{rem}

%


\subsection{Complete graphs} 

We remark that the same LLT polynomial can be realized by different LLT diagrams,
as far as the inversion relations are kept invariant. Keeping this in mind, the simplest LLT diagram corresponding to the complete graph $K_n$ 
is when all the $n$ cells are on the same diagonal which we denote by $\bdnu_{K_n}$. The LLT polynomial of 
$\bdnu_{K_n}$ is known to be modified Hall-Littlewood polynomials indexed by one column shape $\tilde{H}_{(1^n)}(\bx ;q)$. Let us give a simple derivation here.

Consider the \emph{modified Macdonald polynomials} $\tilde{H}_{\mu}({\bf {\bf x}};q,t)$ which have the following expansion in terms of Schur functions 
$$\tilde{H}_{\mu}({\bf x};q,t) = \sum_{\lambda\vdash n}\tilde{K}_{\lambda\mu}(q,t)s_{\lambda},$$
where $\tilde{K}_{\lambda\mu}(q,t)$ are known as \emph{modified $q,t$-Kostka polynomials}. The LLT expansion of $\tilde{H}_{\mu}({\bf x};q,t)$ 
is given in \cite{HHL05}, and especially when the LLT diagram is $\bdnu_{K_n}$, by considering the combinatorial description for the monomial expansion 
of $\tilde{H}_{\mu}({\bf x};q,t)$ given in \cite{HHL05}, it is not very hard to see that 
$$\LLT_{K_n}({\bf x};q)=\tilde{H}_{(n)}({\bf x};q,t).$$
For the details, we refer the readers to \cite[Section 3]{HHL05}. In the case when $\mu=(n)$, the $t$-parameter does not occur and thus we can set $t=0$. 
Also, noting that $\tilde{K}_{\lambda\mu}(q,t)=\tilde{K}_{\lambda\mu'}(t,q)$, we obtain
\begin{align*}
\LLT_{K_n}({\bf x};q)&=\tilde{H}_{(n)}({\bf x};q,0)\\
&=  \sum_{\lambda\vdash n}\tilde{K}_{\lambda,(n)}(q,0)s_{\lambda}\\
&=  \sum_{\lambda\vdash n}\tilde{K}_{\lambda,(1^n)}(0,q)s_{\lambda}\\
&=  \sum_{\lambda\vdash n}\tilde{K}_{\lambda, (1^n)}(q)s_{\lambda}=\tilde{H}_{(1^n)}(\bx ;q),
\end{align*}
where $$\tilde{K}_{\lambda, (1^n)}(q)=\sum_{T\in\SYT(\lambda)}q^{cocharge(T)}.$$
Hence, we have 
$$
\LLT_{K_n}({\bf x};q)= \sum_{\lambda\vdash n}\left(\sum_{T\in\SYT(\lambda)}q^{cocharge(T)}\right) s_{\lambda}.
$$
For the detailed description of the cocharge statistic, see \cite{Hag08}. By considering the way how the cocharge statistic is defined and the fact that 
the Dyck diagram $\pi_{\bdnu_{K_n}}$ has the area sequence $a=(n-1,n-2,\dots, 1,0)$, i.e., $a_i = n-i$, for $1\le i \le n$, we can check that 
the statistic in Definition \ref{def:wt} gives another combinatorial description for the Schur coefficients in this case, namely, 
\begin{equation}\label{eqn:LLT_Kn}
\LLT_{K_n}({\bf x};q)= \sum_{\lambda\vdash n}\left(\sum_{T\in\SYT(\lambda)}q^{wt_{\bdnu_{K_n}}(T)}\right) s_{\lambda}.
\end{equation}

%

\subsection{Path graphs}
In this subsection, we consider the LLT polynomial $\LLT_{P_n}({\bf x};q)$  corresponding to the path graph $P_n$ of order $n$.
Definition \ref{def:wt} also gives a combinatorial description for the Schur coefficients appearing in the Schur expansion of $\LLT_{P_n}({\bf x};q)$.

\begin{prop} \label{prop:path}
Let $P_n$ be the path graph of order $n$. Then
 \[
 \LLT_{P_n}({\bf x};q)= \sum_{\lambda\vdash n}\left(\sum_{T\in\SYT(\lambda)}q^{wt_{\bdnu_{P_n}}(T)}\right) s_{\lambda} \, .
 \]
\end{prop}
\begin{proof}
For a word  $w = w_1w_2 \cdots w_n \in \mathbb{Z}_{>0}^n$,
an index $i \in \{ 1,2, \cdots, n \}$ is said to be a \emph{descent} of $w$ if 
$w_i > w_{i+1}$.
If $\bdnu$ is a unicellular LLT diagram of $n$ cells and $\boldsymbol{S} \in \SSYT{(\bdnu)}$, then $\boldsymbol{S}$ can be regarded as a word of length $n$, say $w({\boldsymbol{S}})$,
and
$\inv(\boldsymbol{S})$ is equal to the number of pairs $(i,j)$ satisfying the following conditions:
\begin{enumerate}
\item[(i)] the cell whose the column labeled by $i$ and the row by $j$ is contained in $\pi_{\bdnu}$, and 
\item[(ii)] $i < j$ but $w({\boldsymbol{S}})_i > w({\boldsymbol{S}})_j $.
\end{enumerate}
In the case where $\bdnu = \bdnu_{P_n}$,
to count $\inv(\boldsymbol{S})$
it is enough to consider
pairs of the form $(i, i+1)$ for $1 \leq i \leq n-1$.
That is, for $\boldsymbol{S} \in \SSYT{(\bdnu_{P_n})}$
\[
\inv(\boldsymbol{S}) = |  \, \{ (i, i+1) : w({\boldsymbol{S}})_i > w({\boldsymbol{S}})_{i+1} \} \, |
\]
which counts the number of descents of $w({\boldsymbol{S}})$.

For each word $w$,
we define $D(w)$ to be the set of all descents of $w$
and $\text{stan}(w)$ its standardization, that is,
$\text{stan}(w)$ is the permutation in $S_n$ obtained by sorting pairs $(w_i, i)$ in lexicographic order.
Then it is well known that 
three kinds of descent sets $D(w)$, $D(\text{stan}(w))$, $D(Q(w))$ are the same,
where $Q(w)$ denote the recording tableau corresponding to $w$ in the procedure of
Robinson-Schensted-Knuth insertion algorithm.

In all, we have
\[
\begin{aligned}
  \LLT_{P_n}({\bf x};q) &= \sum_{\boldsymbol{S} \in \SSYT(\bdnu_{P_n})}  q^{\inv{(\boldsymbol{S}})} \,   x^{\boldsymbol{S}} = \sum_{w \in \mathbb{Z}_{>0}^n} \,  q^{|D(w)|} \,  x^{w} \\
 &=  \sum_{\sigma \in S_n} \, q^{|D(\sigma)|} \left( \sum_{\substack{w \in \mathbb{Z}_{>0}^n \text{ such that }\\ \text{stan}(w)=\sigma}} x^{w}\right) \\
  &=  \sum_{\sigma \in S_n} \, q^{|D(Q(\sigma))|} \left( \sum_{\substack{w \in \mathbb{Z}_{>0}^n \text{ such that }\\ Q(w)=Q(\sigma)}} x^{w}\right) \\
  &= \sum_{\lambda \vdash n} \, \sum_{Q \in \SYT(\lambda)} \, q^{|D(Q)|} \left( \sum_{\substack{w \in \mathbb{Z}_{>0}^n \text{ such that }\\ Q(w)=Q}} x^w \right) \,
  = \sum_{\lambda \vdash n} \, \sum_{Q \in \SYT(\lambda)} \, q^{|D(Q)|} \, s_\lambda \, .
\end{aligned}
\]
Our proposition follows from the fact that the area sequence is 
$a_{\bdnu_{P_n}} = (1, 1, \cdots, 1, 0)$ and thus $wt_{\bdnu_{P_n}}(Q) = \sum_{i \in D(Q)} 1 = |D(Q)|$.
\end{proof}

%


\subsection{Graphs related by linear relations} 


Consider the LLT diagram $\bdnu_{K_n}$ corresponding to the complete graph $K_n$ (for the sake of using the linear relation, we use the left-most figure in Figure \ref{fig:linearrel})
and move the cell on the right most diagonal upward so that 
$k$ cells on the left diagonal are left-below of the moved cell (see the middle figure in Figure \ref{fig:linearrel}).
If we consider the corresponding graphs, moving the cell on the second diagonal removes edges connected to the moved vertex and the 
$k$ vertices going below of it. We denote such a graph by $K_n ^{(k)}$.
Then from Theorem \ref{thm:local}, we obtain the following linear relations.

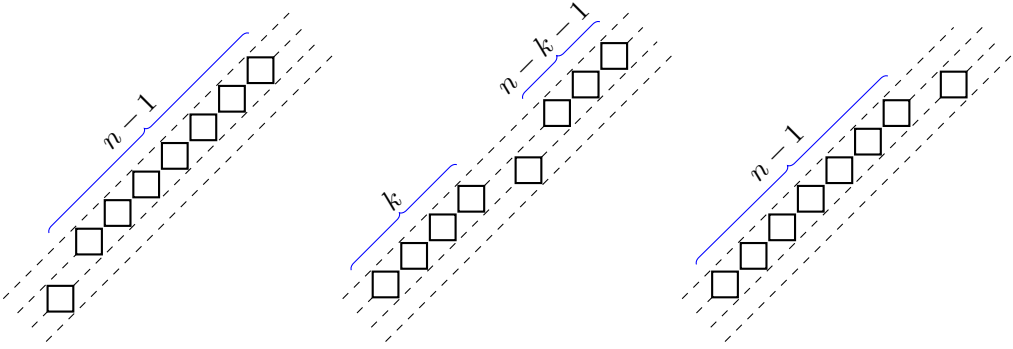
\begin{figure}[ht]
\begin{tikzpicture}[scale=.38]
\draw[thick] (1.05,1.05)--(1.95,1.05)--(1.95,1.95)--(1.05,1.95)--(1.05,1.05)--cycle;
\draw[thick] (2.05,3.05)--(2.95,3.05)--(2.95,3.95)--(2.05,3.95)--(2.05,3.05)--cycle;
\draw[thick] (3.05,4.05)--(3.95,4.05)--(3.95,4.95)--(3.05,4.95)--(3.05,4.05)--cycle;
\draw[thick] (4.05,5.05)--(4.95,5.05)--(4.95,5.95)--(4.05,5.95)--(4.05,5.05)--cycle;
\draw[thick] (5.05,6.05)--(5.95,6.05)--(5.95,6.95)--(5.05,6.95)--(5.05,6.05)--cycle;
\draw[thick] (6.05,7.05)--(6.95,7.05)--(6.95,7.95)--(6.05,7.95)--(6.05,7.05)--cycle;
\draw[thick] (7.05,8.05)--(7.95,8.05)--(7.95,8.95)--(7.05,8.95)--(7.05,8.05)--cycle;
\draw[thick] (8.05,9.05)--(8.95,9.05)--(8.95,9.95)--(8.05,9.95)--(8.05,9.05)--cycle;
\draw[dashed] (-.5,1.5)--(9.5,11.5);
\draw[dashed] (0,1)--(2.05,3.05);
\draw[dashed] (8.95,9.95)--(10,11);
\draw[dashed] (.5,.5)--(1.05,1.05);
\draw[dashed] (1.95,1.95)--(10.5,10.5);
\draw[dashed] (1,0)--(11,10);
\draw [snake=brace,color=blue] (1.1, 3.8)--(8.1, 10.8);
\node (a) at (3.9, 7.7) [rotate=45]{$n-1$};
\end{tikzpicture}
\begin{tikzpicture}[scale=.38]
\draw[thick] (1.05,2.05)--(1.95,2.05)--(1.95,2.95)--(1.05,2.95)--(1.05,2.05)--cycle;
\draw[thick] (2.05,3.05)--(2.95,3.05)--(2.95,3.95)--(2.05,3.95)--(2.05,3.05)--cycle;
\draw[thick] (3.05,4.05)--(3.95,4.05)--(3.95,4.95)--(3.05,4.95)--(3.05,4.05)--cycle;
\draw[thick] (4.05,5.05)--(4.95,5.05)--(4.95,5.95)--(4.05,5.95)--(4.05,5.05)--cycle;
\draw[thick] (6.05,6.05)--(6.95,6.05)--(6.95,6.95)--(6.05,6.95)--(6.05,6.05)--cycle;
\draw[thick] (7.05,8.05)--(7.95,8.05)--(7.95,8.95)--(7.05,8.95)--(7.05,8.05)--cycle;
\draw[thick] (8.05,9.05)--(8.95,9.05)--(8.95,9.95)--(8.05,9.95)--(8.05,9.05)--cycle;
\draw[thick] (9.05,10.05)--(9.95,10.05)--(9.95,10.95)--(9.05,10.95)--(9.05,10.05)--cycle;
\draw[dashed] (0,2)--(10,12);
\draw[dashed] (0.5,1.5)--(1.05,2.05);
\draw[dashed] (4.95,5.95)--(7.05,8.05);
\draw[dashed] (9.95,10.95)--(10.5,11.5);
\draw[dashed] (1,1)--(6.05,6.05);
\draw[dashed] (6.95,6.95)--(11,11);
\draw[dashed] (1.5,0.5)--(11.5,10.5);
\draw [snake=brace,color=blue] (.3, 3)--(4, 6.7);
\draw [snake=brace,color=blue] (6.3, 9)--(9, 11.7);
\node (a) at (1.7, 5.5) [rotate=45]{$k$};
\node (b) at (7, 10.8) [rotate=45]{$n-k-1$};
\end{tikzpicture}
\begin{tikzpicture}[scale=.38]
\draw[thick] (1.05,2.05)--(1.95,2.05)--(1.95,2.95)--(1.05,2.95)--(1.05,2.05)--cycle;
\draw[thick] (2.05,3.05)--(2.95,3.05)--(2.95,3.95)--(2.05,3.95)--(2.05,3.05)--cycle;
\draw[thick] (3.05,4.05)--(3.95,4.05)--(3.95,4.95)--(3.05,4.95)--(3.05,4.05)--cycle;
\draw[thick] (4.05,5.05)--(4.95,5.05)--(4.95,5.95)--(4.05,5.95)--(4.05,5.05)--cycle;
\draw[thick] (5.05,6.05)--(5.95,6.05)--(5.95,6.95)--(5.05,6.95)--(5.05,6.05)--cycle;
\draw[thick] (6.05,7.05)--(6.95,7.05)--(6.95,7.95)--(6.05,7.95)--(6.05,7.05)--cycle;
\draw[thick] (7.05,8.05)--(7.95,8.05)--(7.95,8.95)--(7.05,8.95)--(7.05,8.05)--cycle;
\draw[thick] (9.05,9.05)--(9.95,9.05)--(9.95,9.95)--(9.05,9.95)--(9.05,9.05)--cycle;
\draw[dashed] (0,2)--(9.5,11.5);
\draw[dashed] (0.5,1.5)--(1.05,2.05);
\draw[dashed] (7.95,8.95)--(10.5,11.5);
\draw[dashed] (1,1)--(9.05,9.05);
\draw[dashed] (9.95,9.95)--(11,11);
\draw[dashed] (1.5,.5)--(11.5,10.5);
\draw [snake=brace,color=blue] (.5, 3.2)--(7.2, 9.8);
\node (a) at (3.3, 7.1) [rotate=45]{$n-1$};
\end{tikzpicture}
\caption{$\bdnu_{K_n}$, $\bdnu_{K_n ^{(k)}}$ and $\bdnu_{K_n ^{(n-1)}}$ from the left. }\label{fig:linearrel}
\end{figure}

\begin{prop}\label{prop:lrs}
\begin{equation}
\LLT_{K_n}({\bf x};q) +q[n-2]_q \LLT_{K_n ^{(n-1)}}({\bf x};q) = [n-1]_q \LLT_{K_n ^{(n-2)}}({\bf x};q).\label{eqn:lr1}
\end{equation}
More generally, for $1\le k\le \ell-1$ and $2\le \ell \le n-1$, we have 
\begin{equation}
[\ell -k]_q \LLT_{K_n}({\bf x};q) +q^{\ell -k}[k]_q \LLT_{K_n ^{(\ell)}}({\bf x};q) = [\ell]_q \LLT_{K_n ^{(k)}}({\bf x};q).\label{eqn:lr2}
\end{equation}
\end{prop}

By using the linear relations in Proposition \ref{prop:lrs}, we can prove combinatorial formulas corresponding to LLT diagrams $\bdnu_{K_n ^{(k)}}$.

\begin{prop} 
We have 
$$\LLT _{K_n ^{(k)}} ({\bf x};q)=\sum_{\lambda\vdash n}\left( \sum_{T\in \SYT (\lambda)}q^{wt_{\bdnu_{K_n ^{(k)}}}(T)}\right)s_\lambda,$$
where 
$$wt_{\bdnu_{K_n ^{(k)}}}(T) =\sum_{i\in D(T)} a_i,$$
with the area sequence $a_1 = n-k-1$ and $a_i = n-i$ for  $2\le i\le n$. 
\end{prop}

\begin{proof}
We use the linear relation \eqref{eqn:lr2} when $\ell =n-1$ :
$$
[n-k-1]_q \LLT_{K_n}({\bf x};q) +q^{n -k-1}[k]_q \LLT_{K_n ^{(n-1)}}({\bf x};q) = [n-1]_q \LLT_{K_n ^{(k)}}({\bf x};q).
$$
Considering this linear relation, given $\lambda\vdash n$, for each $T\in\SYT(\lambda)$,  we need to prove that 
$$q^{wt_{\bdnu_{K_n ^{(k)}}}(T)}=\frac{1}{[n-1]_q}\left([n -k-1]_q\cdot q^{wt_{\bdnu_{K_n }}(T)}+q^{n-k-1}[k]_q \cdot q^{wt_{\bdnu_{K_n ^{(n-1)}}}(T)}\right).$$
First of all, observe that $a_i$ values of the Dyck diagrams corresponding to $\bdnu_{K_n}$, $\bdnu_{K_n ^{(k)}}$ and $\bdnu_{K_n ^{(n-1)}}$ are the same, for 
$2\le i\le n$ (or $n\ne 1$) as $a_i = n-i$. 
We divide the cases when $1\in D(T)$ or not. 
If $1 \notin D(T)$, then 
\begin{align*}
&\frac{1}{[n-1]_q}\left([n -k-1]_q\cdot q^{wt_{\bdnu_{K_n }}(T)}+q^{n-k-1}[k]_q \cdot q^{wt_{\bdnu_{K_n ^{(n-1)}}}(T)}\right) \\
&= \frac{q^{wt_{\bdnu_{K_n}}(T)}}{[n-1]_q} ([n-k-1]_q +q^{n-k-1}[k]_q)\\
&= q^{wt_{\bdnu_{K_n}}(T)}=q^{wt_{\bdnu_{K_n ^{(k)}}}(T)}.
\end{align*}
If $1\in D(T)$, then 
\begin{align*}
&\frac{1}{[n-1]_q}\left([n-k-1]_q \cdot q^{wt_{\bdnu_{K_n}}(T)}+q^{n-k-1}[k]_q \cdot q^{wt_{\bdnu_{K_n ^{(n-1)}}}(T)}\right) \\
&= \frac{1}{[n-1]_q} \left([n-k-1]_q \cdot q^{n-1}\cdot q^{wt_{\bdnu_{K_n ^{(n-1)}}}(T)}+q^{n-k-1}[k]_q \cdot q^{wt_{\bdnu_{K_n ^{(n-1)}}}(T)}\right)\\
&=  \frac{q^{n-k-1}\cdot q^{wt_{\bdnu_{K_n ^{(n-1)}}}(T)}}{[n-1]_q}(q^{k}[n-k-1]_q +[k]_q)\\
&=q^{n-k-1} \cdot q^{wt_{\bdnu_{K_n ^{(n-1)}}}(T)}=q^{wt_{\bdnu_{K_n ^{(k)}}}(T)}.
\end{align*}
\end{proof}


\subsection{Lollipop graphs}


In this section, we consider the Schur expansion of LLT polynomials corresponding to lollipop graphs defined in Section \ref{sec:lollipopG}.

\begin{figure}[ht]
\begin{tikzpicture}[scale=.4]
\draw[thick] (.05,.05)--(.95,.05)--(.95,.95)--(.05,.95)--(.05,.05)--cycle;
\draw[thick] (1.05,1.05)--(1.95,1.05)--(1.95,1.95)--(1.05,1.95)--(1.05,1.05)--cycle;
\draw[thick] (2.05,2.05)--(2.95,2.05)--(2.95,2.95)--(2.05,2.95)--(2.05,2.05)--cycle;
\draw[thick] (3.05,3.05)--(3.95,3.05)--(3.95,3.95)--(3.05,3.95)--(3.05,3.05)--cycle;
\draw[thick] (4.05,4.05)--(4.95,4.05)--(4.95,4.95)--(4.05,4.95)--(4.05,4.05)--cycle;
\draw[thick] (7.05,7.05)--(7.95,7.05)--(7.95,7.95)--(7.05,7.95)--(7.05,7.05)--cycle;
\draw[thick] (6.05,5.05)--(6.95,5.05)--(6.95,5.95)--(6.05,5.95)--(6.05,5.05)--cycle;
\draw[thick] (10.05,9.05)--(10.95,9.05)--(10.95,9.95)--(10.05,9.95)--(10.05,9.05)--cycle;
\draw[thick] (9.05,7.05)--(9.95,7.05)--(9.95,7.95)--(9.05,7.95)--(9.05,7.05)--cycle;
\draw[thick] (13.05,11.05)--(13.95,11.05)--(13.95,11.95)--(13.05,11.95)--(13.05,11.05)--cycle;
\draw[thick] (12.05,9.05)--(12.95,9.05)--(12.95,9.95)--(12.05,9.95)--(12.05,9.05)--cycle;
\draw[dashed] (-1,0)--(12.5,13.5);
\draw[dashed] (-.5,-.5)--(.05,.05);
\draw[dashed] (4.95,4.95)--(7.05,7.05);
\draw[dashed] (7.95,7.95)--(13.5,13.5);
\draw[dashed] (0,-1)--(6.05,5.05);
\draw[dashed] (6.95,5.95)--(10.05,9.05);
\draw[dashed] (10.95,9.95)--(14,13);
\draw[dashed] (.5,-1.5)--(9.05,7.05);
\draw[dashed] (9.95,7.95)--(13.05,11.05);
\draw[dashed] (13.95,11.95)--(14.5,12.5);
\draw[dashed] (1,-2)--(12.05,9.05);
\draw[dashed] (12.95,9.95)--(15,12);
\draw[dashed] (1.5,-2.5)--(15.5,11.5);
\node[] at (.5, .5) {$11$};
\node[] at (1.5, 1.5) {$10$};
\node[] at (2.5, 2.5) {$9$};
\node[] at (3.5, 3.5) {$8$};
\node[] at (4.5, 4.5) {$7$};
\node[] at (6.5, 5.5) {$5$};
\node[] at (7.5, 7.5) {$6$};
\node[] at (10.5, 9.5) {$4$};
\node[] at (9.5, 7.5) {$3$};
\node[] at (13.5, 11.5) {$2$};
\node[] at (12.5, 9.5) {$1$};
\end{tikzpicture}
\begin{tikzpicture}[scale=.17]
\foreach \n in {6}
 \foreach \i in {1,...,\n}
    \fill (\i*360/\n:\n) coordinate (n\i) circle(2*\n pt)
      \ifnum \i>1 foreach \j in{\i,...,1}{(n\i)edge(n\j)}\fi;
\draw (180:6)--(-31,0);

\foreach \i in {1,...,5}
 \node[below] at (-36+5*\i,0) {\i}; 

\node[below] at (180:6) {6};
\node[below] at (240:6) {7};
\node[below] at (300:6) {8};
\node[right] at (0:6) {9};
\node[above] at (60:6) {10};
\node[above] at (120:6) {11};

\filldraw [black] (-11,0) circle (12pt)
			(-16,0) circle (12pt)
			(-21,0) circle (12pt)
			(-26,0) circle (12pt)
			(-31,0) circle (12pt);
\end{tikzpicture}
\caption{The LLT diagram corresponding to the lollipop graph $L_{6,5}$}
\end{figure}
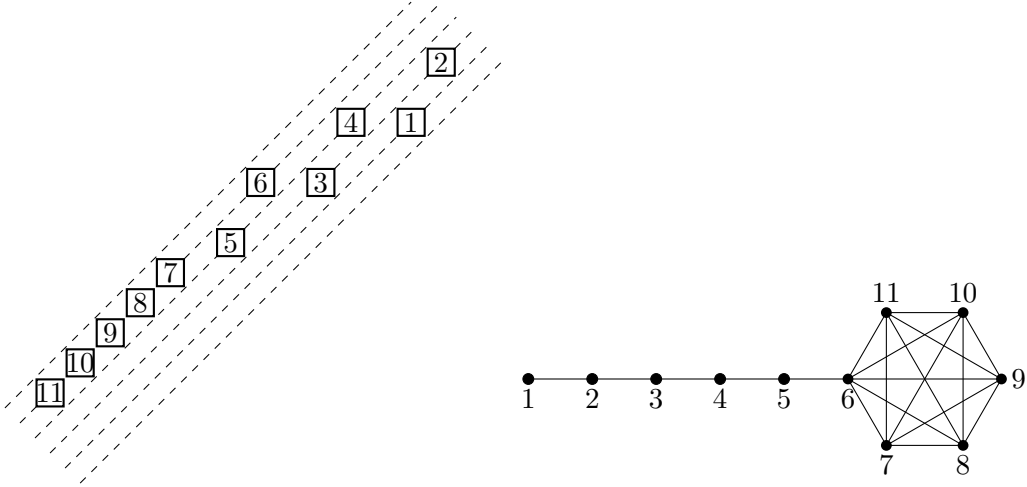

\begin{prop}\label{prop:LLTlollipop}
We have 
$$
\LLT_{L_{m,n}}({\bf x};q) =\sum_{\lambda\vdash n}\left( \sum_{T\in \SYT(\lambda) }q^{wt_{L_{m,n}}(T)}\right) s_{\lambda},
$$
where $wt_{L_{m,n}}(T)=\sum_{i\in D(T)}a_i$ with 
$$a_i = \begin{cases} 
1 & \text{ for } 1\le i \le n,\\
m+n-i & \text{ for } n+1 \le i\le m+n.\end{cases}$$
\end{prop}

\begin{proof}
Let us rewrite the linear relation given in Proposition \ref{prop:quasilollipop} in terms of LLT polynomials ;
\begin{equation}\label{eqn:lollipoplr1}
\LLT_{L_{m,n}}({\bf x};q) =\frac{1}{[m]_q} \left(\LLT_{L_{m+1,n-1}}({\bf x};q) +q[m-1]_q \LLT_{K_m}({\bf x};q) \cdot \LLT_{P_n}({\bf x};q) \right).
\end{equation}
We utilize the above linear relation to prove the Schur coefficients formula. Note that \eqref{eqn:lollipoplr1} can be used to compute the LLT polynomial corresponding to the lollipop graph $L_{m,n}$,
given the LLT polynomial corresponding to the graph $L_{m+1,n-1}$ which has a larger complete graph part and shorter path graph part. So, as an initial case, 
we prove a combinatorial formula for $\LLT_{L_{N-1,1}}({\bf x};q)$. In this case, the linear relation becomes 
\begin{equation}\label{eqn:lollipoplr2}
\LLT_{L_{N-1,1}}({\bf x};q) =\frac{1}{[N-1]_q} \left(\LLT_{K_N}({\bf x};q) +q[N-2]_q \LLT_{K_{N-1}}({\bf x};q) \cdot s_1  \right).
\end{equation}
By Lemma \ref{lem:prod_of_LLT}, we know that 
\begin{align*}
\LLT_{K_{N-1}}({\bf x};q) \cdot s_1  & = \LLT_{K_{N}^{(N-1)}}({\bf x};q)\\
&= \sum_{\lambda\vdash N}\left(\sum_{T\in\SYT (\lambda)} q^{wt_{K_N ^{(N-1)}}(T)} \right)s_\lambda, 
\end{align*}
where $wt_{K_N ^{(N-1)}}(T)=\sum_{i\in D(T)}a_i$ with $a_1 =0$, and $a_i =N-i$ for $2\le i\le N$. 
We already have seen this type of Schur expansion for $\LLT_{K_N}({\bf x};q)$ in \eqref{eqn:LLT_Kn} with 
$wt_{K_N}(T)=\sum_{i\in D(T)}a_i$, for $a_i = N-i$, $1\le i \le N$. Observing that $a_i$ values are consistent for $2\le i\le N$, we consider two cases 
when $1\in D(T)$ and $1\notin D(T)$, and prove, for $\lambda\vdash N$, $T\in \SYT(\lambda)$, 
$$
q^{wt_{L_{N-1,1}}(T)}=\frac{1}{[N-1]_q} \left( q^{wt_{K_N}(T)}+q[N-2]_q \cdot q^{wt_{K_N ^{(N-1)}}(T)}\right).$$
If $1\notin D(T)$, then 
\begin{align*}
& \frac{1}{[N-1]_q} \left( q^{wt_{K_N}(T)}+q[N-2]_q \cdot q^{wt_{K_N ^{(N-1)}}(T)}\right)\\
&=  \frac{ q^{wt_{K_N}(T)}}{[N-1]_q} (1+q[N-2]_q)\\
&= q^{wt_{K_N}(T)} =  q^{wt_{L_{N-1,1}}(T)}.
\end{align*}
If $1\in D(T)$, then 
\begin{align*}
& \frac{1}{[N-1]_q} \left( q^{wt_{K_N}(T)}+q[N-2]_q \cdot q^{wt_{K_N ^{(N-1)}}(T)}\right)\\
&= \frac{1}{[N-1]_q} \left(q^{N-1}\cdot  q^{wt_{K_N ^{(N-1)}}(T)}+q[N-2]_q \cdot q^{wt_{K_N ^{(N-1)}}(T)}\right)\\
&= \frac{ q^{1+ wt_{K_N ^{(N-1)}}(T)}}{[N-1]_q}(q^{N-2}+[N-2]_q)\\
&= q^{1+ wt_{K_N ^{(N-1)}}(T)} = q^{wt_{L_{N-1,1}}(T)}.
\end{align*}
Now, for $\lambda\vdash n$, we compute the coefficient of $s_\lambda$ of the right hand side of \eqref{eqn:lollipoplr1} (denote it by $\mathsf{RHS}$)
and check that it is consistent with $\sum_{T\in \SYT(\lambda) }q^{wt_{L_{m,n}}(T)}$. 
By Lemma \ref{lem:prod_of_LLT}, we know that 
$$
 \LLT_{K_{m}}({\bf x};q) \cdot \LLT_{P_{n}}({\bf x};q)  = \sum_{\lambda\vdash n}\left( \sum_{T\in \SYT(\lambda) }q^{wt_{K_{m} \cup P_{n}}(T)}\right) s_{\lambda},$$
 where $wt_{K_{m} \cup P_{n}}(T)=\sum_{i\in D(T)}a_i$ with 
$$a_i = \begin{cases} 
1 & \text{ for } 1\le i \le n-1,\\
0, & \text{ for } i = n,\\
m+n-i & \text{ for } n+1 \le i\le m+n.\end{cases}$$
So
\begin{align*}
\mathsf{RHS} &= \frac{1}{[m]_q}\left(\sum_{\substack{T\in\SYT(\lambda)\\ n\in D(T)}} q^{m-1}\cdot q^{wt_{L_{m,n}}(T)} +
\sum_{\substack{T\in\SYT(\lambda)\\ n\notin D(T)}}  q^{wt_{L_{m,n}}(T)}\right.\\
& \qquad\qquad\qquad\qquad\qquad\left.+q[m-1]_q \sum_{T\in\SYT(\lambda)}q^{wt_{L_{m,n}}(T)-\chi (n\in D(T))}\right)\\
&=  \frac{1}{[m]_q}\left(\sum_{\substack{T\in\SYT(\lambda)\\ n\in D(T)}} q^{m-1}\cdot q^{wt_{L_{m,n}}(T)} +
\sum_{\substack{T\in\SYT(\lambda)\\ n\notin D(T)}}  q^{wt_{L_{m,n}}(T)} \right.\\
&\quad\qquad\qquad\left.  +  [m-1]_q \sum_{\substack{T\in\SYT(\lambda)\\ n\in D(T)}} q^{wt_{L_{m,n}}(T)} +q[m-1]_q
\sum_{\substack{T\in\SYT(\lambda)\\ n\notin D(T)}}  q^{wt_{L_{m,n}}(T)}\right)\\
&=\frac{1}{[m]_q}\left(\sum_{\substack{T\in\SYT(\lambda)\\ n\in D(T)}} q^{wt_{L_{m,n}}(T)} (q^{m-1}+[m-1]_q)
+ \sum_{\substack{T\in\SYT(\lambda)\\ n\notin D(T)}} q^{wt_{L_{m,n}}(T)}(1+q[m-1]_q) \right)\\
&=\sum_{\substack{T\in\SYT(\lambda)\\ n\in D(T)}} q^{wt_{L_{m,n}}(T)}+ \sum_{\substack{T\in\SYT(\lambda)\\ n\notin D(T)}} q^{wt_{L_{m,n}}(T)}\\
&= \sum_{T\in\SYT(\lambda)} q^{wt_{L_{m,n}}(T)}.
\end{align*}
\end{proof}


\subsection{Melting lollipop graphs}


The Schur coefficients of LLT polynomials corresponding to the melting lollipop graphs (see Section \ref{melting} for the definition of melting lollipop graphs) can be described in a similar fashion.

\begin{prop}\label{prop:meltinglollipop}
We have 
$$
\LLT_{L_{m,n}^{(k)}}({\bf x};q) =\sum_{\lambda\vdash n}\left( \sum_{T\in \SYT(\lambda) }q^{wt_{L_{m,n}^{(k)}}(T)}\right) s_{\lambda},
$$
where $wt_{L_{m,n}^{(k)}}(T)=\sum_{i\in D(T)}a_i$ with 
$$a_i = \begin{cases} 
1 & \text{ for } 1\le i \le n,\\
m-k-1, & \text{ for } i = n+1,\\
m+n-i & \text{ for } n+2 \le i\le m+n.\end{cases}$$
\end{prop}

\begin{proof}
To prove this Schur expansion formula, we rewrite the linear relation  \eqref{eqn:mlollipoplr} in terms of LLT polynomials 
\begin{multline}\label{eqn:mlollipoplr1}
\LLT_{L_{m,n}^{(k)}}({\bf x};q) \\=\frac{1}{[m-1]_q} \left([m-k-1]_q \LLT_{L_{m,n}}({\bf x};q) +q^{m-k-1}[k]_q \LLT_{K_{m-1}}({\bf x};q) \cdot \LLT_{P_{n+1}}({\bf x};q) \right).
\end{multline}
We already obtained the Schur expansion of $\LLT_{L_{m,n}}({\bf x};q)$ in Proposition \ref{prop:LLTlollipop} and by Lemma \ref{lem:prod_of_LLT}, we have 
$$
 \LLT_{K_{m-1}}({\bf x};q) \cdot \LLT_{P_{n+1}}({\bf x};q)  = \sum_{\lambda\vdash n}\left( \sum_{T\in \SYT(\lambda) }q^{wt_{K_{m-1}\cup P_{n+1}}(T)}\right) s_{\lambda},$$
 where $wt_{K_{m-1}\cup P_{n+1}}(T)=\sum_{i\in D(T)}a_i$ with 
$$a_i = \begin{cases} 
1 & \text{ for } 1\le i \le n,\\
0, & \text{ for } i = n+1,\\
m+n-i & \text{ for } n+2 \le i\le m+n.\end{cases}$$
For $\lambda\vdash n$, we compute the coefficient of $s_\lambda$ of the right hand side of \eqref{eqn:mlollipoplr1} 
and check that it is consistent with $\sum_{T\in \SYT(\lambda) }q^{wt_{L_{m,n}^{(k)}}(T)}$. The way how the proof goes is similar to the proof of Proposition \ref{prop:LLTlollipop},
by dividing the cases when $n+1\in D(T)$ and when $n+1\notin D(T)$. We omit the details. 
\end{proof}

\section{Appendix}

In this section, we introduce a combinatorial way to compute the Schur coefficients of LLT polynomials when the Schur functions are indexed by hook shapes. We apply the result of Egge-Loehr-Warrington \cite{ELW} to the 
quasisymmetric expansion of LLT polynomials. Especially when the LLT polynomials are unicellular, then the weight statistic in Definition \ref{def:wt} can be also used in the description of the Schur coefficients indexed by hook shapes. 

We note a quasisymmetric expansion of LLT polynomials given in \cite{HHL05}. A semistandard tableau $\boldsymbol{S}$ is 
\emph{standard} if it is a bijection $\boldsymbol{S}:\bigsqcup \bdnu \rightarrow \{1,2,\dots, n\}$, where 
$n=|\bdnu|=\sum_{j=1}^k |\nu^{(j)}|$. We denote the set of standard tableaux of shape $\bdnu$ by $\SYT(\bdnu)$. 

Define the \emph{descent set} $D(\boldsymbol{S})\subseteq \{1,2,\dots, n-1\}$ of $\boldsymbol{S}\in \SYT(\bdnu)$ by 
$$D(\boldsymbol{S}) =\{ i ~:~ \boldsymbol{S}^{-1}(i+1)\text{ precedes } \boldsymbol{S}^{-1}(i) \text{ in the content reading order} \}.$$
Then, 
\begin{equation}\label{eqn:LLTinQ}
\LLT _{\bdnu}({\bf x};q) =\sum_{\boldsymbol{S}\in \SYT(\bdnu)}q^{\inv (\boldsymbol{S})}F_{co(D(\boldsymbol{S}))}({\bf x}),
\end{equation}
where $co(D(\boldsymbol{S}))$ is the composition corresponding to the set $D(\boldsymbol{S})$ and $F_\alpha ({\bf x})$ is the fundamental 
quasisymmetric function indexed by the composition $\alpha$.


\subsection{Schur coefficients indexed by hook shapes}


Recall the result of Egge-Loehr-Warrington \cite{ELW} on obtaining the Schur expansion given the quasisymmetric expansion in terms of the fundamental quasisymmetric functions.

A skew diagram $\lambda/\mu$ is a \emph{rim-hook} of $\lambda$ if $\lambda/\mu$ does not contain any $2\times 2$ subdiagram and any two consecutive cells of $\lambda/\mu$ share an edge. A rim-hook is \emph{special} if it starts from the cell in the first column. The number of rows of a rim hook $H$ is referred to as its \emph{height}, denoted by $\text{ht}(H)$. The sign of a rim hook $H$ is defined to be $(-1)^{\text{ht}(H)-1}$. A \emph{special rim-hook tableau} $S$ of shape $\lambda$ and content $\alpha$ is a partition of the diagram of $\lambda$ using special rim-hooks such that the length of the $i$th rim-hook from the bottom is $\alpha_i$. The sign of $S$ is the product of the signs of the rim hooks of $S$. 

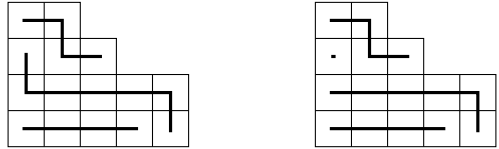
\begin{figure}[ht]
\begin{tikzpicture}[scale=.48]
\draw (0,0)--(0,4)--(2,4)--(2,3)--(3,3)--(3,2)--(5,2)--(5,0)--(0,0)--cycle;
\draw (0,1)--(5,1);
\draw (0,2)--(3,2);
\draw (0,3)--(2,3);
\draw (1,0)--(1,4);
\draw (2,0)--(2,3);
\draw (3,0)--(3,2);
\draw (4,0)--(4,2);
\draw[very thick] (.4,3.5)--(1.5, 3.5)--(1.5,2.5)--(2.6,2.5);
\draw[very thick] (.5, 2.6)--(.5,1.5)--(4.5,1.5)--(4.5,.4);
\draw[very thick] (.4, .5)--(3.6,.5);
\end{tikzpicture}
\qquad\qquad
\begin{tikzpicture}[scale=.48]
\draw (0,0)--(0,4)--(2,4)--(2,3)--(3,3)--(3,2)--(5,2)--(5,0)--(0,0)--cycle;
\draw (0,1)--(5,1);
\draw (0,2)--(3,2);
\draw (0,3)--(2,3);
\draw (1,0)--(1,4);
\draw (2,0)--(2,3);
\draw (3,0)--(3,2);
\draw (4,0)--(4,2);
\draw[very thick] (.4,3.5)--(1.5, 3.5)--(1.5,2.5)--(2.6,2.5);
\draw[very thick] (.44, 2.5)--(.56,2.5);
\draw[very thick] (.4,1.5)--(4.5,1.5)--(4.5,.4);
\draw[very thick] (.4, .5)--(3.6,.5);
\end{tikzpicture}
\caption{Special rim hook tableaux of content $\alpha=(4,7,4)$ on the left and $\alpha=(4,6,1,4)$ on the right.}
\end{figure}

The result of Egge-Loehr-Warrington \cite{ELW} gives a combinatorial description of Schur coefficients, given a fundamental quasisymmetric expansion of any symmetric functions. 

\begin{thm}\cite[Theorem 11]{ELW}\label{thm:ELW1}
Suppose $\mathbb{F}$ is a field, and we have a symmetric function
$$f=\sum_{\lambda\vdash n}c_\lambda s_\lambda =\sum_{\alpha\models n }d_{\alpha}F_\alpha \quad(c_\lambda , d_\alpha\in \mathbb{F}).$$
Then we have 
$$c_\lambda=\sum_{\alpha\models n}d_\alpha K_n ^{\ast} (\alpha, \lambda)$$
for all $\lambda\vdash n$, where 
$$K_n ^{\ast} (\alpha, \lambda) =\sum_{\beta \text{ finer than } \alpha}K_n ' (\beta, \lambda),$$
and $K_n '$ is a right inverse of the Kostka matrix $K_n$ with entries $K'_n(\alpha,\lambda)$, the sum of the signs of the special rim-hook tableaux of shape $\lambda$ and content $\alpha$.
\end{thm}
If each rim-hook contains exactly one cell in the first column of the diagram of $\lambda$, then we say that the rim-hook tableau $S$ of shape $\lambda$ and content $\alpha$ (or equivalently, $(\alpha, \lambda)$) is \textit{flat}. Then we can simplify the description of $K_n ^{\ast}(\alpha, \lambda)$ even more.

\begin{thm}\cite[Theorem 15]{ELW}\label{thm:ELW2}
Let $\alpha\models n$, $\lambda\vdash n$. If $(\alpha,\lambda)$ is flat, then $K_n ^{\ast} (\alpha, \lambda)=K_n '(\alpha,\lambda)=\pm 1$. Otherwise, $K_n ^{\ast}(\alpha, \lambda)=0$. In particular, $K_n ^{\ast}(\alpha,\lambda)=\chi(\alpha=\lambda)$ when $\lambda$ is a hook.
\end{thm}

Given the quasisymmetric expansion \eqref{eqn:LLTinQ}, we apply Theorem \ref{thm:ELW1} and \ref{thm:ELW2} to obtain the Schur coefficients of LLT polynomials when the Schur functions are indexed by hook shapes.

\begin{prop}\label{prop:hook}
Let $\lambda=(k, 1^{n-k})$ be a partition of a hook shape. 
$$\langle \LLT_{\boldsymbol{\nu} }({\bf x};q)   , s_{\lambda} \rangle = \sum_{\substack{\boldsymbol{S}\in\SYT (\boldsymbol{\nu} )  \text{ such that}\\ D(\boldsymbol{S}) = \{ k, k+1,\dots, n-1\}} } q^{\inv(\boldsymbol{S})}.$$
\end{prop}

The Dyck diagram explained in Section \ref{subsec:relation} can be used to compute 
$\langle \LLT_{\boldsymbol{\nu} }({\bf x};q)   , s_{\lambda} \rangle$ in Proposition \ref{prop:hook}. Since $\bdnu$ is an $n$-tuple of single cells, $\boldsymbol{S}\in\SYT (\boldsymbol{\nu} )$ can be considered as a word of length $n$, and 
to satisfy the condition $D(\boldsymbol{S}) = \{ k, k+1,\dots, n-1\}$, the reading words should be in the set of shuffle product of 
$(n-1,n-2,\dots, k+1)$ and $(1,2,\dots, k-1)$ followed by $k$ in the end. 
We denote by $D_k$ the set of all words obtained from such shuffle product.
To compute the inversion statistic, place the reading word in $D_k$ on the main diagonal starting from 
the bottom-left corner of the Dyck diagram and count the number of inversion pairs in this setting. 

\begin{example}\label{ex:hook}
We keep considering the LLT diagram $\bdnu$ in Example \ref{ex:dg}. 
To obtain, for instance, the coefficient of $s_{2111}$, we have to consider the set of reading words $(1\shuffle (5,4,3))2$, i.e., 
\begin{align*}
& 15432\\
&51432\\
&54132\\
&54312.
\end{align*}
We place those reading words on the diagonal of the Dyck diagram and compute the inversion statistic.
$$
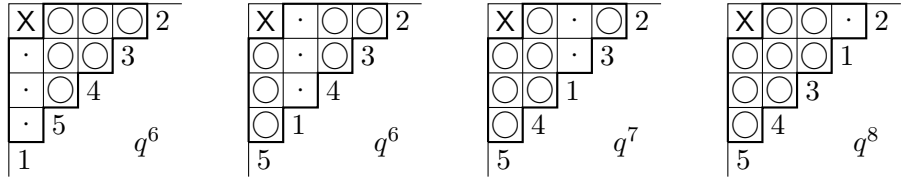
\begin{tikzpicture}[scale=.46]
\draw (0,0)--(0,5)--(5,5);
\draw (0,4)--(4,4);
\draw (0,3)--(3,3);
\draw (0,2)--(2,2);
\draw (0,1)--(1,1);
\draw (1,1)--(1,5);
\draw (2,2)--(2,5);
\draw (3,3)--(3,5);
\draw (4,4)--(4,5);
\draw[thick] (0,1)--(0,4)--(1,4)--(1,5)--(4,5)--(4,4)--(3,4)--(3,3)--(2,3)--(2,2)--(1,2)--(1,1)--(0,1);
\node (7) at (.5, 4.5) {$\mathsf{X}$};
\node (8) at (.45, .5) {$1$};
\node (9) at (1.45, 1.5) {$5$};
\node (10) at (2.45, 2.5) {$4$};
\node (11) at (3.45, 3.5) {$3$};
\node (12) at (4.45, 4.5) {$2$};
\node (13) at (1.5,4.5) {$\bigcirc$};
\node (14) at (2.5,4.5) {$\bigcirc$};
\node (15) at (3.5,4.5) {$\bigcirc$};
\node (16) at (.5,3.5) {$\cdot$};
\node (17) at (1.5,3.5) {$\bigcirc$};
\node (18) at (2.5,3.5) {$\bigcirc$};
\node (19) at (.5,2.5) {$\cdot$};
\node (20) at (1.5,2.5) {$\bigcirc$};
\node (21) at (.5,1.5) {$\cdot$};
\node (22) at (4, 1) {$q^6$};
\end{tikzpicture}
\qquad 
\begin{tikzpicture}[scale=.46]
\draw (0,0)--(0,5)--(5,5);
\draw (0,4)--(4,4);
\draw (0,3)--(3,3);
\draw (0,2)--(2,2);
\draw (0,1)--(1,1);
\draw (1,1)--(1,5);
\draw (2,2)--(2,5);
\draw (3,3)--(3,5);
\draw (4,4)--(4,5);
\draw[thick] (0,1)--(0,4)--(1,4)--(1,5)--(4,5)--(4,4)--(3,4)--(3,3)--(2,3)--(2,2)--(1,2)--(1,1)--(0,1);
\node (7) at (.5, 4.5) {$\mathsf{X}$};
\node (8) at (.45, .5) {$5$};
\node (9) at (1.45, 1.5) {$1$};
\node (10) at (2.45, 2.5) {$4$};
\node (11) at (3.45, 3.5) {$3$};
\node (12) at (4.45, 4.5) {$2$};
\node (13) at (1.5,4.5) {$\cdot$};
\node (14) at (2.5,4.5) {$\bigcirc$};
\node (15) at (3.5,4.5) {$\bigcirc$};
\node (16) at (.5,3.5) {$\bigcirc$};
\node (17) at (1.5,3.5) {$\cdot$};
\node (18) at (2.5,3.5) {$\bigcirc$};
\node (19) at (.5,2.5) {$\bigcirc$};
\node (20) at (1.5,2.5) {$\cdot$};
\node (21) at (.5,1.5) {$\bigcirc$};
\node (22) at (4, 1) {$q^6$};
\end{tikzpicture}
\qquad
\begin{tikzpicture}[scale=.46]
\draw (0,0)--(0,5)--(5,5);
\draw (0,4)--(4,4);
\draw (0,3)--(3,3);
\draw (0,2)--(2,2);
\draw (0,1)--(1,1);
\draw (1,1)--(1,5);
\draw (2,2)--(2,5);
\draw (3,3)--(3,5);
\draw (4,4)--(4,5);
\draw[thick] (0,1)--(0,4)--(1,4)--(1,5)--(4,5)--(4,4)--(3,4)--(3,3)--(2,3)--(2,2)--(1,2)--(1,1)--(0,1);
\node (7) at (.5, 4.5) {$\mathsf{X}$};
\node (8) at (.45, .5) {$5$};
\node (9) at (1.45, 1.5) {$4$};
\node (10) at (2.45, 2.5) {$1$};
\node (11) at (3.45, 3.5) {$3$};
\node (12) at (4.45, 4.5) {$2$};
\node (13) at (1.5,4.5) {$\bigcirc$};
\node (14) at (2.5,4.5) {$\cdot$};
\node (15) at (3.5,4.5) {$\bigcirc$};
\node (16) at (.5,3.5) {$\bigcirc$};
\node (17) at (1.5,3.5) {$\bigcirc$};
\node (18) at (2.5,3.5) {$\cdot$};
\node (19) at (.5,2.5) {$\bigcirc$};
\node (20) at (1.5,2.5) {$\bigcirc$};
\node (21) at (.5,1.5) {$\bigcirc$};
\node (22) at (4, 1) {$q^7$};
\end{tikzpicture}
\qquad
\begin{tikzpicture}[scale=.46]
\draw (0,0)--(0,5)--(5,5);
\draw (0,4)--(4,4);
\draw (0,3)--(3,3);
\draw (0,2)--(2,2);
\draw (0,1)--(1,1);
\draw (1,1)--(1,5);
\draw (2,2)--(2,5);
\draw (3,3)--(3,5);
\draw (4,4)--(4,5);
\draw[thick] (0,1)--(0,4)--(1,4)--(1,5)--(4,5)--(4,4)--(3,4)--(3,3)--(2,3)--(2,2)--(1,2)--(1,1)--(0,1);
\node (7) at (.5, 4.5) {$\mathsf{X}$};
\node (8) at (.45, .5) {$5$};
\node (9) at (1.45, 1.5) {$4$};
\node (10) at (2.45, 2.5) {$3$};
\node (11) at (3.45, 3.5) {$1$};
\node (12) at (4.45, 4.5) {$2$};
\node (13) at (1.5,4.5) {$\bigcirc$};
\node (14) at (2.5,4.5) {$\bigcirc$};
\node (15) at (3.5,4.5) {$\cdot$};
\node (16) at (.5,3.5) {$\bigcirc$};
\node (17) at (1.5,3.5) {$\bigcirc$};
\node (18) at (2.5,3.5) {$\bigcirc$};
\node (19) at (.5,2.5) {$\bigcirc$};
\node (20) at (1.5,2.5) {$\bigcirc$};
\node (21) at (.5,1.5) {$\bigcirc$};
\node (22) at (4, 1) {$q^8$};
\end{tikzpicture}
$$
Hence, we obtain 
$$\langle \LLT_{\bdnu}({\bf x};q), s_{2111}\rangle = 2q^6 +q^7+q^8.$$
\end{example}

%

To describe the Schur coefficients in Proposition \ref{prop:hook} in terms of the weight statistic used in Section \ref{sec:LLT_Schur},
we  first recall the descent set of words and Young tableaux, respectively.
For a word $u = u_1u_2 \cdots u_n$,
we say that an index $i \in \{ 1,2,\cdots,n-1 \}$ is a descent of $u$ if $u_i > u_{i+1}$
and denote by $D(u)$ the set of all descents of $u$.
For a standard Young tableau $T$,
we say that $i \in \{ 1,2,\cdots,n-1 \}$ is a descent of $T$ if $i$ appears in a lower row of $T$ than $i+1$
and denote by $D(T)$ the set of all descents of $T$.
Notice that if $u \in D_k$, then $u_n=k$,
for any $1 \leq i <j \leq k-1$ $j$ cannot precede $i$
and for any $k+1 \leq i <j \leq n$ $i$ cannot precede $j$ in $u$.
Hence,
whenever $u \in D_k$ and $i \in D(u)$, $u_i$ should be in $\{k+1,k+2,\cdots,n-1 \}$
and $u_i > u_j$ for all $i < j$.
This implies that
if $u$ is a reading word of $\boldsymbol{S}\in\SYT (\boldsymbol{\nu} )$ such that $D(\boldsymbol{S}) = \{ k, k+1,\dots, n-1\}$,
then $\inv(\bf{S})$ counts the number of cells whose column is labeled by $i$ and row by $j$ in $\pi_{\boldsymbol{\nu}}$ for all $i \in D(u)$ and all $j <i$.
Letting $\boldsymbol{\nu}^t$ be the LLT diagram whose corresponding Dyck diagram is the conjugate of $\pi_{\boldsymbol{\nu}}$ as a skew shape,
we can see that
\begin{equation} \label{LLT_conjugate_hook}
\langle \LLT_{\boldsymbol{\nu}^t }({\bf x};q)   , s_{\lambda}  \rangle
= \sum_{u \in D_k} q^{wt_{\boldsymbol{\nu}}(u)}
\end{equation}
where $wt_{\boldsymbol{\nu}}(u)$ is the sum of $a_i$ in $a_{\boldsymbol{\nu}}$ for all $i \in D(u)$.

On the other hand,
from the definition of the shuffle product
one can see that every word in $D_k$ can be obtained from the word $1 \, 2  \cdots (k-1) \, n \, (n-1) \cdots (k+1) \, k $ by applying the following relation in succession:
replace $xzy$ by $zxy$ or vice versa if $x < y < z$.
This suggests that any two words in $D_k$ are Knuth equivalent
and thus
they result the same insertion tableau in the procedure of Robinson-Shensted-Knuth (RSK) insertion algorithm.
Moreover,
since when $\lambda = (k, 1^{n-k})$ $\SYT(\lambda)$ and $D_k$ are equicardinal,
RSK algorithm guarantees that
$\SYT(\lambda)$ is the set of all recording tableaux $Q(u)$ for $u \in D_k$.
Finally,
by combining \eqref{LLT_conjugate_hook} with well known facts that
$\LLT_{\boldsymbol{\nu}^t }({\bf x};q) = \LLT_{\boldsymbol{\nu} }({\bf x};q)$
and $i \in D(u)$ if and only if $i \in D(Q(u))$,
we can conclude the following proposition.

\begin{prop}\label{prop:hook_unicellular}
Let $\lambda=(k, 1^{n-k})$ be a partition of a hook shape and
$\boldsymbol{\nu}$ a unicellular LLT diagram.
If $a_{\boldsymbol{\nu}} = (a_1, a_2, \cdots, a_n)$ is the area sequence of $\pi_{\boldsymbol{\nu}}$, then
$$\langle \LLT_{\boldsymbol{\nu} }({\bf x};q)   , s_{\lambda} \rangle =
\sum_{T \in \SYT(\lambda)} q^{wt_{\boldsymbol{\nu}}(T)} \, , $$
where $wt_{\boldsymbol{\nu}}(T) = \sum_{i \in D(T)} a_i$.
\end{prop}

\begin{example} \label{exam:hook_with_descent}
We keep considering the LLT diagram $\bdnu$
and the hook partition $\lambda$ in Example \ref{ex:hook}.
Then we have the area sequence $a_{\bdnu} = (3,3,2,1,0)$
and the following four standard Young tableaux of shape $(2,1,1,1)$
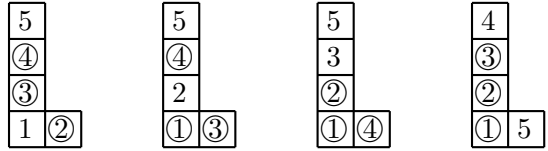
\begin{center}
\begin{tikzpicture}[scale=.48]
\draw[thick] (0,0)--(2,0);
\draw[thick] (0,0)--(0,4);
\draw[thick] (0,1)--(2,1);
\draw[thick] (0,2)--(1,2);
\draw[thick] (0,3)--(1,3);
\draw[thick] (0,4)--(1,4);
\draw[thick] (1,0)--(1,4);
\draw[thick] (2,0)--(2,1);
\node (1) at (.45, .5) {$1$};
\node (2) at (1.45,.5) {$2$};
\node (6) at (1.45,.5) {$\bigcirc$};
\node (3) at (.45, 1.5) {$3$};
\node (7) at (.45, 1.5) {$\bigcirc$};
\node (4) at (.45, 2.5) {$4$};
\node (8) at (.45, 2.5) {$\bigcirc$};
\node (5) at (.45, 3.5) {$5$};
\end{tikzpicture} \qquad
\begin{tikzpicture}[scale=.48]
\draw[thick] (0,0)--(2,0);
\draw[thick] (0,0)--(0,4);
\draw[thick] (0,1)--(2,1);
\draw[thick] (0,2)--(1,2);
\draw[thick] (0,3)--(1,3);
\draw[thick] (0,4)--(1,4);
\draw[thick] (1,0)--(1,4);
\draw[thick] (2,0)--(2,1);
\node (1) at (.45, .5) {$1$};
\node (8) at (.45, .5) {$\bigcirc$};
\node (2) at (1.45,.5) {$3$};
\node (6) at (1.45,.5) {$\bigcirc$};
\node (3) at (.45, 1.5) {$2$};
\node (4) at (.45, 2.5) {$4$};
\node (8) at (.45, 2.5) {$\bigcirc$};
\node (5) at (.45, 3.5) {$5$};
\end{tikzpicture} \qquad
\begin{tikzpicture}[scale=.48]
\draw[thick] (0,0)--(2,0);
\draw[thick] (0,0)--(0,4);
\draw[thick] (0,1)--(2,1);
\draw[thick] (0,2)--(1,2);
\draw[thick] (0,3)--(1,3);
\draw[thick] (0,4)--(1,4);
\draw[thick] (1,0)--(1,4);
\draw[thick] (2,0)--(2,1);
\node (1) at (.45, .5) {$1$};
\node (8) at (.45, .5) {$\bigcirc$};
\node (2) at (1.45,.5) {$4$};
\node (6) at (1.45,.5) {$\bigcirc$};
\node (3) at (.45, 1.5) {$2$};
\node (7) at (.45, 1.5) {$\bigcirc$};
\node (4) at (.45, 2.5) {$3$};
\node (5) at (.45, 3.5) {$5$};
\end{tikzpicture} \qquad
\begin{tikzpicture}[scale=.48]
\draw[thick] (0,0)--(2,0);
\draw[thick] (0,0)--(0,4);
\draw[thick] (0,1)--(2,1);
\draw[thick] (0,2)--(1,2);
\draw[thick] (0,3)--(1,3);
\draw[thick] (0,4)--(1,4);
\draw[thick] (1,0)--(1,4);
\draw[thick] (2,0)--(2,1);
\node (1) at (.45, .5) {$1$};
\node (8) at (.45, .5) {$\bigcirc$};
\node (2) at (1.45,.5) {$5$};
\node (3) at (.45, 1.5) {$2$};
\node (7) at (.45, 1.5) {$\bigcirc$};
\node (4) at (.45, 2.5) {$3$};
\node (8) at (.45, 2.5) {$\bigcirc$};
\node (5) at (.45, 3.5) {$4$};
\end{tikzpicture}
\end{center}
with its $\boldsymbol{\nu}$-weight 6,6,7 and 8, respectively.
Once again, we obtain
$$\langle \LLT_{\bdnu}({\bf x};q), s_{2111}({\bf x})\rangle = 2q^6 +q^7+q^8.$$
\end{example}



\section*{Acknowledgements}\label{sec:acknow} 
The authors would like to thank Soojin Cho for her kind support and encouragement.


\vspace{3mm}

\end{document}